\documentclass[dvipdfm,letterpaper,leqno]{article}

\usepackage{amsmath,amsthm,amsfonts,amssymb}
\usepackage{euscript,mathrsfs,empheq}
\usepackage{enumitem,color}
\usepackage[authoryear,longnamesfirst]{natbib}
\usepackage{hyperref}
\usepackage{anysize}
\usepackage{verbatim}
\usepackage{mathtools}

\mathtoolsset{showonlyrefs,showmanualtags}

\definecolor{darkblue}{rgb}{0.03,0.03,0.23}
\hypersetup{dvipdfm,pdfauthor=Daniel Remenik,colorlinks=true,urlcolor=darkblue,citecolor=darkblue,linkcolor=darkblue}

\marginsize{3.75cm}{3.75cm}{3cm}{3.5cm}

\newtheorem{teo}{Theorem}
\newtheorem{lem}{Lemma}[section]
\newtheorem{teosec}[lem]{Theorem}
\newtheorem{prop}[lem]{Proposition}
\newtheorem{cor}[lem]{Corollary}
\newtheorem{teotype}{Theorem}
\theoremstyle{definition}

\newtheorem{ex}[lem]{Example}
\newtheorem{remark}[lem]{Remark}
\newtheorem{assum}{Assumption}

\newcommand{\pp}{\mathbb{P}}

\newcommand{\ee}{\mathbb{E}}
\newcommand{\rr}{\mathbb{R}}

\newcommand{\nn}{\mathbb{N}}
\newcommand{\zz}{\mathbb{Z}}
\newcommand{\cc}{\mathcal{C}}

\newcommand{\cp}{\mathcal{P}}

\newcommand{\chu}{{\mathcal{H}_1}}
\newcommand{\chup}{{\chu^\prime}}
\newcommand{\chd}{{\mathcal{H}_2}}
\newcommand{\chdp}{{\chd^{\prime}}}
\newcommand{\cht}{{\mathcal{H}_3}}
\newcommand{\chtp}{{\cht^{\prime}}}
\newcommand{\chf}{{\mathcal{H}_4}}
\newcommand{\chfp}{{\chf^{\prime}}}
\newcommand{\chii}{{\mathcal{H}_i}}
\newcommand{\chip}{{\chii^{\prime}}}
\newcommand{\ccz}{{\mathcal{C}_0}}
\newcommand{\cczp}{{\ccz^\prime}}

\newcommand{\ccd}{{\mathcal{C}_2}}
\newcommand{\ccdp}{{\ccd^\prime}}
\newcommand{\cct}{{\mathcal{C}_3}}
\newcommand{\cctp}{{\cct^\prime}}
\newcommand{\cci}{{\mathcal{C}_i}}
\newcommand{\ccip}{{\cci^\prime}}
\newcommand{\ogamma}{{\overline{\gamma}}}
\newcommand{\olambda}{{\overline{\lambda}}}

\newcommand{\cbw}{\cc_b(W)}

\newcommand{\mw}{\mathcal{B}(W)}

\newcommand{\uno}[1]{\mathbf{1}_{#1}}
\newcommand{\br}[2]{\left<{#1},{#2}\right>}
\newcommand{\dif}{(\varphi(w')-\varphi(w))}
\newcommand{\difd}{(\varphi(w_1')+\varphi(w_2')-\varphi(w_1)-\varphi(w_2))}

\newcommand{\ootimes}{\hspace{-1pt}\otimes\hspace{-1pt}}
\newcommand{\pparagraph}[1]{\vspace{-10pt}\paragraph{#1}}
\newcommand{\ep}{\varepsilon}

\newcommand{\norm}[1]{\left\|#1\right\|}

\newcounter{sys}
\renewcommand{\thesys}{S\arabic{sys}}

\numberwithin{equation}{section}

\let\oldmarginpar\marginpar
\renewcommand{\marginpar}[1]{\oldmarginpar{\raggedright\texttt{#1}}}

\allowdisplaybreaks[1]

\begin{document}

\title{Limit Theorems for Individual-Based Models in Economics and Finance
  \footnotetext{\\[-8pt]\emph{MSC}: 60K35, 60B12, 46N30, 62P50, 91B70\\
    \emph{Keywords}: Individual-based model, interacting particle system, law of large
    numbers, central limit theorem, fluctuations process, measure-valued process, finance, economics} }

\author{Daniel Remenik\vspace{2pt}\\Cornell University}

\maketitle

\thispagestyle{empty}

\begin{abstract}
  There is a widespread recent interest in using ideas from statistical physics to model
  certain types of problems in economics and finance. The main idea is to derive the
  macroscopic behavior of the market from the random local interactions between
  agents. Our purpose is to present a general framework that encompasses a broad range of models,
  by proving a law of large numbers and a central limit theorem for certain interacting
  particle systems with very general state spaces. To do this we draw inspiration from
  some work done in mathematical ecology and mathematical physics. The first result is
  proved for the system seen as a measure-valued process, while to prove the second one we
  will need to introduce a chain of embeddings of some abstract Banach and Hilbert spaces
  of test functions and prove that the fluctuations converge to the solution of a certain
  generalized Gaussian stochastic differential equation taking values in the dual of one
  of these spaces.
\end{abstract}

\section{Introduction}\label{sec:intr}

We consider interacting particle systems of the following form. There is a fixed number
$N$ of particles, each one having a type $w\in W$. The particles change their types via
two mechanisms. The first one corresponds simply to transitions from one type to another
at some given rate. The second one involves a direct interaction between particles: pairs
of particles interact at a certain rate and acquire new types according to some given
(random) rule. We will allow these rates to depend directly on the types of the particles
involved and on the distribution of the whole population on the type space.

Our purpose is to prove limit theorems, as the number of particles $N$ goes to infinity,
for the empirical random measures $\nu^N_t$ associated to these systems. $\nu^N_t$ is
defined as follows: if $\eta^N_t(i)\in W$ denotes the type of the $i$-th particle at time
$t$, then
\[\nu^N_t=\frac{1}{N}\sum_{i=1}^N\delta_{\eta^N_t(i)},\]
where $\delta_w$ is the probability measure on $W$ assigning mass 1 to $w$.

Our first result, Theorem \ref{thm:limit}, provides a law of large numbers for $\nu^N_t$
on a finite time interval $[0,T]$: the empirical measures converge in distribution to a
deterministic continuous path $\nu_t$ in the space of probability measures on $W$, whose
evolution is described by a certain system of integro-differential equations. Theorem
\ref{thm:clt} analyzes the fluctuations of the finite system $\nu^N_t$ around $\nu_t$, and
provides an appropriate central limit result: the fluctuations are of order $1/\sqrt{N}$,
and the asymptotic behavior of the process $\sqrt{N}\big(\nu^N_t-\nu_t\big)$ has a
Gaussian nature. This second result is, as could be expected, much more delicate than the
first one.

In recent years there has been an increasing interest in the use of interacting particle
systems to model phenomena outside their original application to statistical physics, with
special attention given to models in ecology, economics, and finance. Our model is
specially suited for the last two types of problems, in particular because we have assumed
a constant number of particles, which may represent agents in the economy or financial
market (ecological problems, on the other hand, usually require including birth and death
of particles). Particle systems were first used in this context in \citet{foll}, and they
have been used recently by many authors to analyze a variety of problems in economics and
finance. The techniques that have been used are diverse, including, for
instance, ideas taken from the Ising model in \citet{foll}, the voter model in
\citet{gieWeb}, the contact process in \citet{huckKos}, the theory of large deviations in
\citet{daiPra}, and the theory of queuing networks in \citet{davisEsp} and
\citet{bayHoSir}.
 
Our original motivation for this work comes precisely from financial modeling. It is
related to some problems studied by Darrell Duffie and coauthors (see Examples
\ref{ex:otc} and \ref{ex:infPerc}) in which they derive some models from the random local
interactions between the financial agents involved, based on the ideas of \citet{dufSun}.
Our initial goal was to provide a general framework in which this type of problems could
be rigorously analyzed, and in particular prove a law of large numbers for them. In our
general setting, $W$ will be allowed to be any locally compact complete separable metric
space. Considering type spaces of this generality is one of the main features of our
model, and it allows us to provide a unified framework to deal with models of different
nature (for instance, the model in Example \ref{ex:otc} has a finite type space and the
limit solves a finite system of ordinary differential equations, while in Example
\ref{ex:infPerc} the type space is $\rr$ and the limit solves a system of uncountably many
integro-differential equations).

To achieve this first goal, we based our model and techniques on ideas taken from the
mathematical biology literature, and in particular on \citet{fourMel}, where the authors
study a model that describes a spatial ecological system where plants disperse seeds and
die at rates that depend on the local population density, and obtain a deterministic limit
similar to ours. We remark that, following their ideas, our results could be extended to
systems with a non-constant population by adding assumptions which allow to control the
growth of the population, but we have preferred to keep this part of the problem
simple.

The central limit result arose as a natural extension of this original question,
but, as we already mentioned, it is much more delicate. The extra technical difficulties
are related with the fact that the fluctuations of the process are signed measures (as
opposed to the process $\nu^N_t$ which takes values in a space of probability measures), and
the space of signed measures is not well suited for the study of convergence in
distribution. The natural topology to consider for this space in our setting, that of weak
convergence, is in general not metrizable. One could try to regard this space as the Banach space
dual of the space of continuous bounded functions on $W$ and endow it with its operator
norm, but this topology is too strong in general to obtain tightness for the fluctuations (observe
that, in particular, the total mass of the fluctuations $\sqrt{N}\big(\nu^N_t-\nu_t\big)$
is not a priori bounded uniformly in $N$). To overcome this difficulty we will show convergence
of the fluctuations as a process taking values in the dual of a suitable abstract Hilbert
space of test functions. We will actually have to consider a sequence of embeddings of
Banach and Hilbert spaces, which will help us in controlling the norm of the fluctuations. This approach
is inspired by ideas introduced in \citet{metiv} to study weak convergence of some
measure-valued processes using sequences of Sobolev embeddings. Our proof is based on
\citet{meleard}, where the author proves a similar central limit result for a system of
interacting diffusions associated with Boltzmann equations.

The rest of the paper is organized as follows. Section \ref{sec:model} contains the
description of the general model, Section \ref{sec:lln} presents the law of large numbers
for our system, and Section \ref{sec:clt} presents the central limit theorem, together
with the description of the extra assumptions and the functional analytical setting we
will use to obtain it. All the proofs are contained in Section \ref{sec:proofs}.

\section{Description of the Model}\label{sec:model}

\subsection{Introductory example}

To introduce the basic features of our model and fix some ideas, we begin by presenting one of the basic
examples we have in mind.

\begin{ex}\label{ex:otc}
We consider the model for over-the-counter markets introduced in \citet{dufGarPed}. There
is a ``consol'', which is an asset paying dividends at a constant rate of 1, and there are
$N$ investors that can hold up to one unit of the asset. The total number of units of the
asset remains constant in time, and the asset can be traded when the investors contact
each other and when they are contacted by marketmakers. Each investor is characterized by
whether he or she owns the asset or not, and by an intrinsic type that is ``high'' or
``low''. Low-type investors have a holding cost when owning the asset, while high-type
investors do not. These characteristics will be represented by the set of types
$W=\{ho,hn,lo,ln\}$, where \emph{h} and \emph{l} designate the high- and low-type of an
investor while \emph{o} and \emph{n} designate whether an investor owns or not the asset.

At some fixed rate $\lambda_d$, high-type investors change their type to low. This means
that each investor runs a Poisson process with rate $\lambda_d$ (independent from the
others), and at each event of this process the investor changes his or her intrinsic type
to low (nothing happens if the investor is already of low-type).  Analogously, low-type
investors change to high-type at some rate $\lambda_u$.  The meetings between agents are
defined as follows: each investor decides to look for another investor at rate $\beta$
(understood as before, i.e., at the times of the events of a Poisson process with rate
$\beta$), chooses the investor uniformly among the set of $N$ investors, and tries to
trade. Additionally, each investor contacts a marketmaker at rate $\rho$. The marketmakers
pair potential buyers and sellers, and the model assumes that this pairing happens
instantly. At equilibrium, the rate at which investors trade through marketmakers is
$\rho$ times the minimum between the fraction of investors willing to buy and the fraction
of investors willing to sell (see \citet{dufGarPed} for more details). In this model, the
only encounters leading to a trade are those between $hn$- and $lo$-agents, since
high-type investors not owning the asset are the only ones willing to buy, while low-type
investors owning the asset are the only ones willing to sell.

Theorem \ref{thm:limit} will imply the following for this model: as $N$
goes to infinity, the (random) evolution of the fraction of agents of each type
converges to a deterministic limit which is the unique solution of the following system of
ordinary differential equations:
\begin{equation}\label{eq:otc}
  \begin{alignedat}{2}
    \dot{u}_{ho}(t)&=&2\beta u_{hn}(t)u_{lo}(t)+\rho\min\{u_{hn}(t),u_{lo}(t)\}+\lambda_uu_{lo}(t)-\lambda_du_{ho}(t),\\
    \dot{u}_{hn}(t)&=&-2\beta u_{hn}(t)u_{lo}(t)-\rho\min\{u_{hn}(t),u_{lo}(t)\}+\lambda_uu_{ln}(t)-\lambda_du_{hn}(t),\\
    \dot{u}_{lo}(t)&=&-2\beta u_{hn}(t)u_{lo}(t)-\rho\min\{u_{hn}(t),u_{lo}(t)\}-\lambda_uu_{lo}(t)+\lambda_du_{ho}(t),\\
    \dot{u}_{ln}(t)&=&2\beta u_{hn}(t)u_{lo}(t)+\rho\min\{u_{hn}(t),u_{lo}(t)\}-\lambda_uu_{ln}(t)+\lambda_du_{hn}(t).\\
  \end{alignedat}
\end{equation}
Here $u_{w}(t)$ denotes the fraction of type-$w$ investors at time $t$. This deterministic
limit corresponds to the one proposed in \citet{dufGarPed} for this model (see the
referred paper for the interpretation of this equations and more on this model).
\end{ex}

\subsection{Description of the General Model}\label{subsec:general}

We will denote by $I_N=\{1,\dotsc,N\}$ the set of particles in the system. In line with
our original financial motivation, we will refer to these particles as the ``agents'' in
the system (like the investors of the aforementioned example).  The possible types for the
agents will be represented by a locally compact Polish (i.e., separable, complete,
metrizable) space $W$. Given a metric space $E$, $\cp(E)$ will denote the collection of
probability measures on $E$, which will be endowed with the topology of weak
convergence. When $E=W$, we will simply write $\cp=\cp(W)$. We will denote by $\cp_a$ the
subset of $\cp$ consisting of purely atomic measures.

The Markov process $\nu^N_t$ we are interested in takes values in $\cp_a$ and describes
the evolution of the distribution of the agents over the set of types. We recall that it
is defined as
\[\nu^N_t=\frac{1}{N}\sum_{i=1}^N\delta_{\eta^N_t(i)},\]
where $\delta_w$ is the probability measure on $W$ assigning mass 1 to $w\in W$ and
$\eta^N_t(i)$ corresponds to the type of the agent $i$ at time $t$. In other words, the
vector $\eta^N_t\in W^{I_N}$ gives the configuration of the set of agents at time $t$,
while for any Borel subset $A$ of $W$, $\nu^N_t(A)$ is the fraction of agents whose type
is in $A$ at time $t$.

The dynamics of the process is defined by the following rates:
\begin{itemize}[topsep=2pt,itemsep=-2pt]
\item Each agent decides to change its type at a certain rate $\gamma(w,\nu^N_t)$ that
  depends on its current type $w$ and the current distribution $\nu^N_t$.  The new type is
  chosen according to a probability measure $a(w,\nu^N_t,dw')$ on $W$.
\item Each agent contacts each other agent at a certain rate that depends on their current
  types $w_1$ and $w_2$ and the current distribution $\nu^N_t$: the total rate at which a
  given type-$w_1$ agent contacts type-$w_2$ agents is given by
  $N\lambda(w_1,w_2,\nu^N_t)\nu^N_t(\{w_2\})$. After a pair of agents meet, they choose
  together a new pair of types according to a probability measure
  $b(w_1,w_2,\nu^N_t,dw_1'\ootimes dw_2')$ (not necessarily symmetric in $w_1,w_2$) on $W\!\times\!W$.
\end{itemize}
For a fixed $\mu\in\cp_a$, $a(w,\mu,dw')$ and $b(w_1,w_2,\mu,dw_1'\ootimes dw_2')$ can be
interpreted, respectively, as the transition kernels of Markov chains in $W$ and $W\!\times\!W$.

Let $\mw$ be the collection of bounded measurable functions on $W$ and $\cbw$ be the
collection of bounded continuous functions on $W$. For $\nu\in\cp$ and $\varphi\in\mw$
(or, more generally, any measurable function $\varphi$) we write
\[\br{\nu}{\varphi}=\int_W\!\varphi\,d\nu.\]
Observe that
\[\br{\nu^N_t}{\varphi}=\frac{1}{N}\sum_{i=1}^N\varphi(\eta^N_t\!(i)).\]

We make the following assumption:

\begin{assum}\label{assum:1}
  \mbox{}
  \begin{enumerate}[label=(\ref*{assum:1}\arabic{*}),ref=\ref*{assum:1}\arabic{*},topsep=2pt]
  \item\label{assum:1:bd} The rate functions $\gamma(w,\nu)$ and $\lambda(w,w',\nu)$ are
    defined for all $\nu\in\cp$. They are non-negative, measurable in $w$ and $w'$,
    bounded respectively by constants $\ogamma$ and
    $\olambda$, and continuous in $\nu$.
  \item\label{assum:1:meas} $a(w,\nu,\cdot)$ and $b(w,w',\nu,\cdot)$ are measurable in $w$ and $w'$.
  \item\label{assum:1:lip} The mappings
    \begin{subequations}
      \begin{align}
        \nu&\longmapsto\int_W\!\gamma(w,\nu)\,a(w,\nu,\cdot)\,\nu(dw)\qquad\text{and}
        \label{eq:totalRate1}\\
        \nu&\longmapsto\int_W\!\int_W\!\lambda(w_1,w_2,\nu)\,b(w_1,w_2,\nu,\cdot)\,\nu(dw_2)\,\nu(dw_1),
        \label{eq:totalRate2}
      \end{align}
    \end{subequations}
    which assign to each $\nu\in\cp_a$ a finite measure on $W$ and $W\!\times\!W$,
    respectively, are continuous with respect to the topology of weak convergence and
    Lipschitz with respect to the total variation norm: there are constants
    $C_a,C_b>0$ such that
    \begin{multline*}
      \left\|\int_W\!\gamma(w,\nu_1)a(w,\nu_1,\cdot)\,\nu_1(dw)
        -\int_W\!\gamma(w,\nu_2)a(w,\nu_2,\cdot)\,\nu_2(dw)\right\|_\text{TV}
      \leq C_a\|\nu_1-\nu_2\|_\text{TV}
    \end{multline*}
    and
    \begin{multline*}
      \Bigg\|\int_W\!\int_W\!\lambda(w_1,w_2,\nu_1)b(w_1,w_2,\nu_1,
      \cdot)\,\nu_1(dw_2)\,\nu_1(dw_1)\\
      -\int_W\!\int_W\!\lambda(w_1,w_2,\nu_2)b(w_1,w_2,\nu_2,
      \cdot)\,\nu_2(dw_2)\,\nu_2(dw_1)\Bigg\|_\text{TV}
      \leq C_b\|\nu_1-\nu_2\|_\text{TV}.
    \end{multline*}
  \end{enumerate}
\end{assum}

We recall that the total variation norm of a signed measure $\mu$ is defined by
\[\norm\mu_\text{TV}=\sup_{\varphi:\,\norm\varphi_\infty\leq1}\left|\br{\mu}{\varphi}\right|.\]
\eqref{assum:1:lip} is satisfied, in particular, whenever the rates do not depend on
$\nu$.

\section{Law of large numbers for \texorpdfstring{$\nu^N_t$}{\textbackslash
    nu\textcircumflex N\textunderscore t}}\label{sec:lln}

Our first result shows that the process $\nu^N_t$ converges in distribution, as the number
of agents $N$ goes to infinity, to a deterministic limit that is characterized by a
measure-valued system of differential equations (written in its weak form).

Given a metric space $S$, we will denote by $D([0,T],S)$ the space of c\`adl\`ag functions
$\nu\!:[0,T]\longrightarrow S$, and we endow these spaces with the Skorohod topology (see
\citet{ethKur} or \citet{bill} for a reference on this topology and weak convergence in
general). Observe that our processes $\nu^N_t$ have paths on $D([0,T],\cp)$ (recall that
we are endowing $\cp$ with the topology of weak convergence, which is metrizable). We will
also denote by $C([0,T],S)$ the space of continuous functions $\nu\!:[0,T]\longrightarrow
S$.

\begin{teo}\label{thm:limit}
  Suppose that Assumption \ref{assum:1} holds. For any given $T>0$, consider the sequence
  of $\cp$-valued processes $\nu^N_t$ on $[0,T]$, and assume that the sequence of initial
  distributions $\nu^N_0$ converges in distribution to some fixed $\nu_0\in\cp$.  Then the
  sequence $\nu^N_t$ converges in distribution in $D([0,T],\cp)$ to a deterministic
  $\nu_t$ in $C([0,T],\cp)$, which is the unique solution of the following system of
  integro-differential equations: for every $\varphi\in\mw$ and $t\in[0,T]$,
  \begin{equation}\label{eq:intDiffSys}\tag{\thesys}
    \begin{aligned}
      \br{\nu_t}{\varphi}=\br{\nu_0}{\varphi}&+\int_0^t\!\int_W\!\gamma(w,\nu_s)
      \int_W\!\dif\,a(w,\nu_s,dw')\,\nu_s(dw)\,ds\\
      &+\int_0^t\!\int_W\!\int_W\!\lambda(w_1,w_2,\nu_s)\int_{W\!\times\!W}\!\difd\\
      &\hspace{1.4in}\cdot b(w_1,w_2,\nu_s,dw_1'\ootimes
      dw_2')\,\nu_s(dw_2)\,\nu_s(dw_1)\,ds.
    \end{aligned}
  \end{equation}
\end{teo}

Observe that, in particular,
\eqref{eq:intDiffSys} implies that for every Borel set $A\subseteq W$ and almost every
$t\in[0,T]$,
\begin{equation}\label{eq:diffSys}\tag{\thesys${}^\prime$}
  \begin{aligned}
    \hspace{0\in}\frac{d\nu_t(A)}{dt}&=-\int_A\!\left(\gamma(w,\nu_t)
      +\int_W\!\big(\lambda(w,w',\nu_t)+\lambda(w',w,\nu_t)\big)\nu_t(dw')\right)\nu_t(dw)\\
    &\hspace{0.0in}+\int_W\!\gamma(w,\nu_t)a(w,\nu_t,A)\,\nu_t(dw)\\
    &\hspace{0.0in}+\int_W\!\int_W\!\lambda(w,w',\nu_t)\Big[b(w,w',\nu_t,A\!\times\!W)+b(w,w',\nu_t,W\!\times\!A)\Big]
    \,\nu_t(dw')\,\nu_t(dw).
  \end{aligned}
\end{equation}
Furthermore, standard measure theory arguments allow to show that the system
\eqref{eq:diffSys} actually characterizes the solution of \eqref{eq:intDiffSys} (by
approximating the test functions $\varphi$ in \eqref{eq:intDiffSys} by simple functions).

\eqref{eq:diffSys} has an intuitive interpretation: the first term on the right side is
the total rate at which agents leave the set of types $A$, the second term is the rate at
which agents decide to change their types to a type in $A$, and the third term is the rate
at which agents acquire types in $A$ due to interactions between them.

The following corollary of the previous result is useful when writing and analyzing the
limiting equations \eqref{eq:intDiffSys} or \eqref{eq:diffSys} (see, for instance, Example
\ref{ex:infPerc}).

\begin{cor}\label{cor:absCont}
  In the context of Theorem \ref{thm:limit}, assume that $\nu_0$ is absolutely continuous
  with respect to some measure $\mu$ on $W$ and that the measures
  \[\int_W\!\gamma(w,\nu_0)a(w,\nu_0,\cdot)\,\nu_0(dw)
  \quad\text{and}\quad\int_W\!\int_W\!\lambda(w_1,w_2,\nu_0)b(w_1,w_2,\nu_0,\cdot)\,\nu_0(dw_1)\,\nu_0(dw_2)\]
  are absolutely continuous with respect to $\mu$ and $\mu\ootimes\mu$, respectively. Then
  the limit $\nu_t$ is absolutely continuous with respect to $\mu$ for all $t\in[0,T]$.
\end{cor}

The following two examples show two different kinds of models: one with a finite type
space and the other with $W=\rr$. The first model is the one given in Example \ref{ex:otc}.

\begin{ex}[Continuation of Example \ref{ex:otc}]\label{ex:otc2}
  To translate into our framework the model for over-the-counter markets of
  \citet{dufGarPed}, we take $W=\{ho,hn,lo,ln\}$
  and consider a set of parameters $\gamma$, $a$, $\lambda$, and $b$ with all but
  $\lambda$ being independent of $\nu^N_t$. Let
  \begin{align*}
    \gamma(ho)=\gamma(hn)&=\lambda_d,&a(ho,\cdot)&=\delta_{lo},&a(hn,\cdot)&=\delta_{ln},\\
    \gamma(lo)=\gamma(ln)&=\lambda_u,&a(lo,\cdot)&=\delta_{ho},&a(ln,\cdot)&=\delta_{hn}.
  \end{align*}
  Observe that with this definition, high-type investors become low-type at rate
  $\lambda_d$ and low-type investors become high-type at rate $\lambda_u$, just as
  required. For the encounters between agents we take
  \begin{gather*}
    \lambda(hn,lo,\nu)=\lambda(lo,hn,\nu)=
    \begin{cases}
      \beta+\frac{\rho}{2}\frac{\nu(\{hn\})\wedge\nu(\{lo\})}{\nu(\{hn\})\nu(\{lo\})} &
      \text{if
        $\nu(\{hn\})\nu(\{lo\})>0$,}\\
      \beta & \text{if $\nu(\{hn\})\nu(\{lo\})=0$,}
    \end{cases}\\
    b(hn,lo,\nu,\cdot)=\delta_{(ho,ln)},\qquad\text{and}\qquad
    b(lo,hn,\nu,\cdot)=\delta_{(ln,ho)}
  \end{gather*}
  (where $a\wedge b=\min\{a,b\}$), and for all other pairs $w_1,w_2\in W$,
  $\lambda(w_1,w_2,\nu)=0$ (recall that the only encounters leading to a trade are those
  between $hn$- and $lo$-agents and vice versa, in which case trade always occurs). The
  rates $\lambda(hn,lo,\nu)$ and $\lambda(lo,hn,\nu)$ have two terms: the rate $\beta$
  corresponding to the rate at which $hn$-agents contact $lo$-agents, plus a second rate
  reflecting trades carried out via a marketmaker. The form of this second rate assures
  that $hn$- and $lo$- agents meet through marketmakers at the right rate of
  $\rho\,\nu(\{hn\})\wedge\nu(\{lo\})$. It is not difficult to check that these
  parameters satisfy Assumption \ref{assum:1}, using the fact that
  $x\wedge y=(x+y-|x-y|)/2$ for $x,y\in\rr$.

  Now let $u_w(t)=\nu_t(\{w\})$, where $\nu_t$ is the limit of $\nu^N_t$ given by Theorem
  \ref{thm:limit}. We need to compute the right side of \eqref{eq:diffSys} with $A=\{w\}$
  for each $w\in W$. Take, for example, $w={ho}$. We get
  \begin{equation}\label{eq:otc1}
    \dot{u}_{ho}(t)=\lambda_uu_{lo}(t)-\lambda_du_{ho}(t)+\beta
    u_{hn}(t)u_{lo}(t)+\frac{\rho}{2}u_{hn}(t)\wedge u_{lo}(t)
    +\beta u_{lo}(t)u_{hn}(t)+\frac{\rho}{2}u_{hn}(t)\wedge u_{lo}(t),
  \end{equation}
  which corresponds exactly to the first equation in \eqref{eq:otc}. The other three
  equations follow similarly.
\end{ex}

\begin{ex}\label{ex:infPerc}
  Our second example is based on the model for information percolation in large markets
  introduced in \citet{dufMan}. We will only describe the basic features of the model, for
  more details see the cited paper. There is a random variable $X$ of concern to all
  agents which has two possible values, ``high'' or ``low''. Each agent holds some
  information about the outcome of $X$, and this information is summarized in a real
  number $x$ which is a sufficient statistic for the posterior probability assigned by the
  agent (given his or her information) to the outcome of $X$ being high. We take these
  statistics as the types of the agents (so $W=\rr$). The model is set up so that these
  statistics satisfy the following: after a type-$x_1$ agent and a type-$x_2$ agent meet
  and share their information, $x_1+x_2$ becomes a sufficient statistic for the posterior
  distribution of $X$ assigned by both agents given now their shared information.

  In this model the agents change types only after contacting other agents, so we take
  $\gamma\equiv0$, and encounters between agents occur at a constant rate $\lambda>0$.  The
  transition kernel for the types of the agents after encounters is independent of
  $\nu^N_t$ and is given by
  \[b(x_1,x_2,\cdot)=b(x_2,x_1,\cdot)=\delta_{(x_1+x_2,x_1+x_2)}\] for every
  $x_1,x_2\in\rr$. This choice for the parameters trivially satisfies Assumption
  \ref{assum:1}.

  To compute the limit of the process, let $A$ be a Borel subset of $\rr$. Then, since
  $\gamma\equiv0$ and $\lambda$ is constant, \eqref{eq:diffSys} gives
  \begin{align*}
    \dot{\nu}_t(A)&=-2\lambda\nu_t(A)+\lambda\int_{\rr^2}\!\left(\delta_{(x+y,x+y)}(\rr\!\times\!A)
      +\delta_{(x+y,x+y)}(A\!\times\!\rr)\right)\nu_t(dy)\,\nu_t(dx)\\
    &=-2\lambda\nu_t(A)+2\lambda\int_{\rr^2}\!\delta_{x+y}(A)\,\nu_t(dy)\,\nu_t(dx)
    =-2\lambda\nu_t(A)+2\lambda\int_{-\infty}^\infty\!\nu_t(A-x)\,\nu_t(dx),
  \end{align*}
  where $A-x=\{y\in\rr\!:\,y+x\in A\}$. Therefore,
  \begin{equation}
    \dot{\nu}_t(A)=-2\lambda\nu_t(A)+2\lambda(\nu_t\!\ast\!\nu_t)(A).\label{eq:infPerc}
  \end{equation}

  Using Corollary \ref{cor:absCont} we can write the last equation in a nicer form: if we
  assume that the initial condition $\nu_0$ is absolutely continuous with respect to the
  Lebesgue measure, then the measures $\nu_t$ have a density $g_t$ with respect to the
  Lebesgue measure, and we obtain
  \begin{equation}
    \dot{g}_t(x)=-2\lambda g_t(x)+2\lambda\int_{-\infty}^\infty\!g_t(z-x)g_t(z)\,dz=-2\lambda g_t(x)+2\lambda
    (g_t\!\ast\!g_t)(x).
  \end{equation}
  This is the system of integro-differential equations proposed in \citet{dufMan} for this
  model (except for the factor of $2$, which is omitted in that paper).
\end{ex}

\section{Central limit theorem for \texorpdfstring{$\nu^N_t$}{\textbackslash
    nu\textcircumflex N\textunderscore t}}\label{sec:clt}
Theorem \ref{thm:limit} gives the law of large numbers for $\nu^N_t$, in the sense that it
obtains a deterministic limit for the process as the size of the market goes to
infinity. We will see now that, under some additional hypotheses, we can also obtain a
central limit result for our process: the fluctuations of $\nu^N_t$ around the limit
$\nu_t$ are of order $1/\sqrt{N}$, and they have, asymptotically, a Gaussian nature. As we
mentioned in the \nameref{sec:intr}, this result is much more delicate than Theorem
\ref{thm:limit}, and we will need to work hard to find the right setting for it.

The \emph{fluctuations} process is defined as follows:
\[\sigma^N_t=\sqrt{N}\big(\nu^N_t-\nu_t\big).\]
$\sigma^N_t$ is a sequence of finite signed measures, and our goal is to prove that it
converges to the solution of a system of stochastic differential equations driven by a
Gaussian process. As we explained in the \nameref{sec:intr}, regarding the fluctuations
process as taking values in the space of signed measures, and endowing this space with the
topology of weak convergence (which corresponds to seeing the process as taking values in
the Banach space dual of $\cbw$ topologized with the weak${}^*$ convergence) is not the
right approach for this problem.  The idea will be to replace the test function space
$\cbw$ by an appropriate Hilbert space $\chu$ and regard $\sigma^N_t$ as a linear
functional acting on this space via the mapping
$\varphi\longmapsto\br{\sigma^N_t}{\varphi}$. In other words, we will regard $\sigma^N_t$
as a process taking values in the dual $\chup$ of a Hilbert space $\chu$.

The space $\chu$ that we choose will depend on the type space $W$. Actually, whenever $W$
is not finite we will not need a single space, but a chain of seven spaces embedded in a
certain structure.  Our goal is to handle (at least) the following four possibilities for
$W$: a finite set, $\zz^d$, a ``sufficiently smooth'' compact subset of $\rr^d$, and all
of $\rr^d$. We wish to handle these cases under a unified framework, and this will
require us to abstract the necessary assumptions on our seven spaces and the parameters of
the model. We will do this in Sections \ref{sec:genSett} and \ref{sec:statemCLT}, and then
in Section \ref{sec:applTypeSpaces} we will explain how to apply this abstract setting to
the four type spaces $W$ that we just mentioned.

\subsection{General setting}\label{sec:genSett}

During this and the next subsection we will assume as given the
collection of spaces in which our problem will be embedded, and then we will make
some assumptions on the parameters of our process that will assure that they are compatible
with the structure of these spaces. The idea of this part is that we will try to impose as
little as possible on these spaces, leaving their definition to be specified for the
different cases of type space $W$.

The elements we will use are the following:
\begin{itemize}[itemsep=-2pt,topsep=2pt]
\item Four separable Hilbert spaces of measurable functions on $W$, $\chu$, $\chd$,
  $\cht$, and $\chf$.
\item Three Banach spaces of continuous functions on $W$, $\ccz,\ccd$, and $\cct$.
\item Five continuous functions
  $\rho_0,\rho_1,\rho_2,\rho_3,\rho_4\!:W\longrightarrow[1,\infty)$ such that
  $\rho_i\leq\rho_{i+1}$ for $i=0,1,2,3$, $\rho_i\in\cci$ for $i=0,2,3$, and for all $w\in W$,
  $\rho^p_1(w)\leq C\rho_4(w)$ for some $C>0$ and $p>1$ (this last requirement is very mild, as we will see
  in the examples below, but will be necessary in the proof of Theorem \ref{thm:clt}).
\end{itemize}

The seven spaces and the five functions introduced above must be related in a specific
way. First, we assume that the following sequence of continuous embeddings holds:
\addtocounter{assum}{1}
\begin{equation}
  \label{eq:normIneq}\tag{\Alph{assum}1}
  \ccz\xhookrightarrow[]{\quad}\chu\xhookrightarrow[\text{c}]{\quad}
  \chd\xhookrightarrow[]{\quad}\ccd\xhookrightarrow[]{\quad}
  \cht\xhookrightarrow[]{\quad}\cct\xhookrightarrow[]{\quad}
  \chf,
\end{equation}
where the c under the second arrow means that the embedding is compact.  We recall that a
continuous embedding $E_1\hookrightarrow E_2$ between two normed spaces $E_1,E_2$ implies,
in particular, that $\|\cdot\|_{E_2}\leq C\|\cdot\|_{E_1}$ for some $C>0$, while saying
that the embedding is compact means that every bounded set in $E_1$ is compact in $E_2$.

Second, we assume that for $i=1,2,3,4$, if $\varphi\in\chii$ then
\begin{equation}
  \left|\varphi(w)\right|\leq C\norm\varphi_\chii\rho_i(w)
  \label{assum:rho:h}\tag{\Alph{assum}2}
\end{equation}
for all $w\in W$, for some $C>0$ which does not depend on $\varphi$.  The same holds for
the spaces $\cci$: for $i=0,2,3$ and $\varphi\in\cci$,
\begin{equation}
  \left|\varphi(w)\right|\leq C\norm\varphi_\cci\rho_i(w).
  \label{assum:rho:c}\tag{\Alph{assum}3}
\end{equation}
The functions $\rho_i$ will typically appear as weighting functions in the definition of the norms
of the spaces $\chii$ and $\cci$. They will dictate the maximum growth rate allowed for functions in
these spaces.

We will denote by $\chip$ and $\ccip$ the topological duals of the spaces $\chii$ and
$\cci$, respectively, endowed with their operator norms (in particular, the spaces $\chip$
and $\ccip$ are Hilbert and Banach spaces themselves). Observe that \eqref{eq:normIneq} implies the following
dual continuous embeddings:
\begin{equation}\label{eq:normIneqD}\tag{\Alph{assum}1${}^\prime$}
  \chfp\xhookrightarrow[]{\quad}\cctp\xhookrightarrow[]{\quad}
  \chtp\xhookrightarrow[]{\quad}\ccdp\xhookrightarrow[]{\quad}
  \chdp\xhookrightarrow[\text{c}]{\quad}\chup\xhookrightarrow[]{\quad}
  \cczp.
\end{equation}

Before continuing, let us describe briefly the main ideas behind the proof of our central
limit theorem, which will help explain why this is a good setting for proving convergence
of the fluctuations process. What we want to prove is that $\sigma^N_t$ converges in
distribution, as a process taking values in $\chup$, to the solution $\sigma_t$ of a
certain stochastic differential equation (see \eqref{eq:clt} below). The approach we will
take to prove this (the proof of Theorem \ref{thm:limit} follows an analogous line) is
standard: we first prove that the sequence $\sigma^N_t$ is tight, then we show that any
limit point of this sequence satisfies the desired stochastic differential equation, and
finally we prove that this equation has a unique solution (in distribution). To achieve
this we will follow the line of proof of \citet{meleard}. Our sequence of embeddings
\eqref{eq:normIneqD} corresponds there to a sequence of embeddings of weighted Sobolev
spaces (see (3.11) in the cited paper); in particular, we will use a very similar sequence
of spaces to deal with the case $W=\rr^d$ in Section \ref{sec:type:rrd}. One important
difficulty with this approach is the following: the operator $J_s$ associated with the
drift term of our limiting equation (see \eqref{eq:defJ}), as well as the corresponding
operators $J^N_s$ for $\sigma^N_t$ (see \eqref{eq:defJN}), cannot in general be taken to
be bounded as operators acting on any of the spaces $\chii$. This forces us to introduce
the spaces $\cci$, on which \eqref{assum:rho:c} plus some assumptions on the rates of the
process will assure that $J_s$ and $J^N_s$ are bounded.

The scheme of proof will be roughly as follows. We will consider the semimartingale
decomposition of the real-valued process $\br{\sigma^N_t}{\varphi}$, for $\varphi\in\chf$,
and then show that the martingale part defines a martingale in $\chfp$. This, together
with a moment estimate on the norm of the martingale part in $\chfp$ and the boundedness
of the operators $J^N_s$ in $\cctp$, will allow us to deduce that $\sigma^N_t$ can be seen
as a semimartingale in $\chtp$, and moreover give its semimartingale decomposition. Next,
we will give a uniform estimate (in $N$) of the norm of $\sigma^N_t$ in $\ccdp$. This
implies the same type of estimate in $\chdp$, and this will allow us to obtain the
tightness of $\sigma^N_t$ in $\chup$. The fact that the embedding
$\chdp\hookrightarrow\chup$ is compact is crucial in this step. Then we will show that all
limit points of $\sigma^N_t$ have continuous paths in $\chup$ and they all satisfy the
desired stochastic differential equation \eqref{eq:clt}. Unfortunately, it will not be
possible to achieve this last part in $\chup$, due to the unboundedness of $J_s$ in this
space. Consequently, we are forced to embed the equation in the (bigger) space
$\cczp$. The boundedness of $J_s$ in $\cczp$ will also allow us to obtain uniqueness for
the solutions of this equation in this space, thus finishing the proof.

Our last assumption (\ref{assum:2clt} below) will assure that our abstract setting is
compatible with the rates defining our process. Before that, we need to replace Assumptions
\eqref{assum:1:bd} and \eqref{assum:1:meas} by stronger versions:

\begin{assum}\label{assum:1clt}
  \mbox{}
  \begin{enumerate}[label=(\ref*{assum:1clt}\arabic{*}),ref=\ref*{assum:1clt}\arabic{*},topsep=2pt]
  \item\label{assum:1clt:1} There is a family of finite measures
    $\big\{\Gamma(w,z,\cdot)\big\}_{w,z\in W}$ on $W$, whose total masses are bounded by
    $\ogamma$, such that for every $w\in W$ and every $\nu\in\cp$ we have
    \[\gamma(w,\nu)a(w,\nu,dw')=\int_W\!\Gamma(w,z,dw')\,\nu(dz).\]
    $\Gamma(w,z,\cdot)$ is measurable in $w$ and continuous in $z$.
  \item There is a family of measures $\big\{\Lambda(w_1,w_2,z,\cdot)\big\}_{w_1,w_2,z\in
      W}$ on $W\!\times\!W$, whose total masses are bounded by $\olambda$, such that for
    every $w_1,w_2\in W$ and every $\nu\in\cp$ we have
    \[\lambda(w_1,w_2,\nu)b(w_1,w_2,\nu,dw_1'\ootimes
    dw_2')=\int_{W\!\times\!W}\!\Lambda(w_1,w_2,z,dw_1'\ootimes dw_2')\,\nu(dz).\]
    $\Lambda(w_1,w_2,z,\cdot)$ is measurable in $w_1$ and $w_2$ and continuous in $z$.
  \end{enumerate}
\end{assum}

The intuition behind this assumption is the following: the total rate at which a type-$w$
agent becomes a type-$w'$ agent is computed by averaging the effect that each agent in the
market has on this rate for the given agent. Observe that, under this assumption,
\eqref{assum:1:lip} holds.

\begin{remark}
  Assumption \ref{assum:1clt} has the effect of linearizing the jump rates in $\nu$. This
  turns out to be very convenient, because it will allow us to express the drift term of
  the stochastic differential equation describing the limiting fluctuations $\sigma_t$
  (\eqref{eq:clt} below) as $J_t\sigma_t$ for some $J_t\in\cczp$ (see \eqref{eq:defJ} and
  \eqref{eq:defJN}).  A more general approach would be to assume that the jump kernels,
  seen as operators acting on $\cczp$, are Fr\'echet differentiable. In that case we would
  need to change the form of the drift operator $J_t$ in the limiting equation and of
  Assumption \ref{assum:2clt} below, but the proof of Theorem \ref{thm:clt} would still
  work, without any major modifications. To avoid extra complications, and since all the
  examples we have in mind satisfy Assumption \ref{assum:1clt}, we will restrict ourselves
  to this simpler case.
\end{remark}

We introduce the following notation: given a measurable function $\varphi$ on $W$, let
\begin{gather*}
  \Gamma\varphi(w;z)=\int_W\!\dif\,\Gamma(w,z,dw')\qquad\text{and}\\
  \Lambda\varphi(w_1,w_2;z)=\int_{W\!\times\!W}\!\difd\,\Lambda(w_1,w_2,z,dw_1'\ootimes
  dw_2').
\end{gather*}
These quantities can be thought of as the jump kernels for the process associated with
the effect of a type-$z$ agent on the transition rates. Averaging these rates with respect
to $\nu^N_t(dz)$ gives the total jump kernel for the process.

\begin{assum}\label{assum:2clt}
  \mbox{}
  \begin{enumerate}[label=(\ref*{assum:2clt}\arabic*),ref=\ref*{assum:2clt}\arabic*,topsep=2pt]
  \item\label{assum:2clt:1} There is a $C>0$ such that for all $w,z\in W$ and $i=0,1,2,3,4$,
    \[\int_W\!\rho^2_i(w')\,\Gamma(w,z,dw')<C\left(\rho^2_i(w)+\rho^2_i(z)\right).\]
  \item\label{assum:2clt:2} There is a $C>0$ such that for all $w_1,w_2,z\in W$ and $i=0,1,2,3,4$,
    \[\int_{W\!\times\!W}\!\left(\rho^2_i(w_1')+\rho^2_i(w_2')\right)
    \,\Lambda(w_1,w_2,z,dw_1'\ootimes
    dw_2')<C\left(\rho^2_i(w_1)+\rho^2_i(w_2)+\rho^2_i(z)\right).\]
  \item\label{assum:2clt:3} Let $\mu_1,\mu_2\in\cp$ be such that
    $\br{\mu_i}{\rho^2_4}<\infty$ and define the following operator acting on measurable
    functions $\varphi$ on $W$:
    \begin{align*}
      \hspace{-0.1in}J_{\mu_1,\mu_2}\varphi(z)=&\int_W\!\Gamma\varphi(w;z)\,\mu_1(dw)
      +\int_W\!\Gamma\varphi(z;x)\,\mu_2(dx)\\
      &+\int_W\!\int_W\!\Lambda\varphi(w_1,w_2;z)\,
      \mu_1(dw_2)\,\mu_1(dw_1)
      +\int_W\!\int_W\!\Lambda\varphi(w,z;x)\,\mu_1(dw)\,\mu_2(dx)\\
      &+\int_W\!\int_W\!\Lambda\varphi(z,w;x)\,\mu_2(dw)\,\mu_2(dx).
    \end{align*}
    Then:
    \begin{enumerate}[label=(\roman*),ref=\ref*{assum:2clt:3}.\roman*,topsep=2pt]
    \item\label{assum:2clt:3i} $J_{\mu_1,\mu_2}$ is a bounded operator on $\cci$,
      for $i=0,2,3$. Moreover, its norm can be bounded uniformly in $\mu_1,\mu_2$.
    \item\label{assum:2clt:3ii} There is a $C>0$ such that given any $\varphi\in\ccz$ and
      any $\mu_1,\mu_2,\mu_3,\mu_4\in\cp$ satisfying $\br{\mu_i}{\rho^2_4}<\infty$,
      \[\norm{\left(J_{\mu_1,\mu_2}-J_{\mu_3,\mu_4}\right)\varphi}_\ccz
      \leq
      C\norm\varphi_\ccz\left(\norm{\mu_1-\mu_3}_\ccdp+\norm{\mu_2-\mu_4}_\ccdp\right).\]
    \end{enumerate}
  \end{enumerate}
\end{assum}

\eqref{assum:2clt:1} and \eqref{assum:2clt:2} correspond to moment assumptions on the
transition rates of the agents, and assure that the agents do not jump ``too
far''. \eqref{assum:2clt:3i} says two things: first, that the jump kernel defined by the
rates preserves the structure of the spaces $\cci$ and, second, that the resulting
operator is bounded, which will imply the boundedness of the drift operators $J_s$ and
$J^N_s$ mentioned above. \eqref{assum:2clt:3ii} involves a sort of strengthening of the
Lipschitz condition \eqref{assum:1:lip} on the rates, and will be used to prove uniqueness
for the limiting stochastic differential equation. Observe that by taking larger weighting
functions $\rho_i$, which corresponds to taking smaller spaces of test functions $\chii$,
we add more moment assumptions on the rates of the process; on the other hand, asking for
more structure on the spaces $\chii$ and $\cci$, such as differentiability in the
Euclidean case, adds more requirements on the regularity of the rates.

\subsection{Statement of the theorem}\label{sec:statemCLT}

For $\xi\in\chip$ (respectively $\ccip$) and $\varphi\in\chii$ (respectively $\cci$) we
will write
\[\br{\xi}{\varphi}=\xi(\varphi).\]

Given $\varphi\in\chu$ and $z\in W$ define
\begin{equation}\label{eq:defJ}
\begin{aligned}
  J_s\varphi(z)=\int_W\!\Gamma\varphi(w;z)\,\nu_s(dw)&+\int_W\!\Gamma\varphi(z;x)\,\nu_s(dx)\\
  &+\int_W\!\int_W\!\Lambda\varphi(w_1,w_2;z)\,\nu_s(dw_2)\,\nu_s(dw_1)\\
  &+\int_W\!\int_W\!\big[\Lambda\varphi(z,w;x)+\Lambda\varphi(w,z;x)\big]\,\nu_s(dw)\,\nu_s(dx)
\end{aligned}
\end{equation}
 Observe that $J_s=J_{\nu_s,\nu_s}$. Therefore, under moment
assumptions on $\nu_s$, \eqref{assum:2clt:3i} implies that $J_s$ is a bounded operator on
each of the spaces $\cci$. Observe that if we integrate the first and third terms on the
right side of \eqref{eq:defJ} with respect to $\nu_s(dz)$, we obtain the integral term in
\eqref{eq:intDiffSys}. In our central limit result, the variable $z$ will be integrated
against the limiting fluctuation process $\sigma_t$. The other two terms in
\eqref{eq:defJ} correspond to fluctuations arising from the dependence of the rates on its
other arguments (the types of the agents involved).

The operator $J_s$ (or, more properly, its adjoint $J_s^*$) will appear in the drift term
of the stochastic differential equation describing the limiting fluctuations process,
which will be expressed as a Bochner integral. We recall that these integrals are an
extension of the Lebesgue integral to functions taking values on a Banach space, see
Section V.5 in \citet{yosida} for details.

\begin{teo}\label{thm:clt}
  Assume that Assumptions \ref{assum:1clt} and \ref{assum:2clt} hold, that
  \eqref{eq:normIneq}, \eqref{assum:rho:h}, and \eqref{assum:rho:c} hold, and that
  \begin{equation}
    \begin{gathered}
      \sqrt{N}(\nu^N_0-\nu_0)\Longrightarrow\sigma_0,\qquad
      \sup_{N\geq0}\,\ee\!\left(\left\|\sqrt{N}\left(\nu^N_0
            -\nu_0\right)\right\|^2_\ccdp\right)<\infty,\\
      \sup_{N\geq0}\,\ee\!\left(\br{\nu^N_0}{\rho^2_4}\right)<\infty,\qquad\text{and}\qquad
      \ee\!\left(\br{\nu_0}{\rho^2_4}\right)<\infty
    \end{gathered}\label{eq:iniFluc}
  \end{equation}
  hold, where the convergence in distribution above is in $\chup$. Then the sequence of
  processes $\sigma^N_t$ converges in distribution in $D([0,T],\chup)$ to a process
  $\sigma_t\in C([0,T],\chup)$. This process is the unique (in distribution) solution in
  $\cczp$ of the following stochastic differential equation:
  \begin{equation}\label{eq:clt}\tag{\thesys}
    \sigma_t=\sigma_0+\int_0^t\!J_s^*\sigma_s\,ds+Z_t,
    \stepcounter{sys}
  \end{equation}
  where the above is a Bochner integral, $J_s^*$ is the adjoint of the operator $J_s$ in
  $\ccz$, and $Z_t$ is a centered $\cczp$-valued Gaussian process with quadratic
  covariations specified by
  \begin{equation}\label{eq:coQuad}
    \begin{aligned}
      \big[Z_\cdot(\varphi_1),Z_\cdot(\varphi_2)\big]_t
      =\int_0^t&\!\int_W\!\int_W\!\int_W\!(\varphi_1(w')-\varphi_1(w))(\varphi_2(w')-\varphi_2(w))\,\Gamma(w,z,dw')\\
      &\hspace{2.5in}\cdot\nu_s(dz)\,\nu_s(dw)\,ds\\
      &+\int_0^t\!\int_W\!\int_W\!\int_W\!\int_{W\!\times\!W}\!
      (\varphi_1(w_1')+\varphi_1(w_2')-\varphi_1(w_1)-\varphi_1(w_2))\\
      &\hspace{1.3in}\cdot(\varphi_2(w_1')+\varphi_2(w_2')-\varphi_2(w_1)-\varphi_2(w_2))\\
      &\hspace{0.8in}\cdot\Lambda(w_1,w_2,z,dw_1'\ootimes dw_2')
      \,\nu_s(dz)\,\nu_s(dw_2)\,\nu_s(dw_1)\,ds
    \end{aligned}
  \end{equation}
  for every $\varphi_1,\varphi_2\in\ccz$.
\end{teo}

We will denote by $C^{\varphi_1,\varphi_2}_s$ the sum of the two terms inside the time
integrals above, so
\begin{equation}
  \big[Z_\cdot(\varphi_1),Z_\cdot(\varphi_2)\big]_t=\int_0^t\!C^{\varphi_1,\varphi_2}_s\,ds.\label{eq:defCs}
\end{equation}

\begin{remark}\label{rem:clt}
  \mbox{}
  \begin{enumerate}[label=\arabic*.,topsep=2pt,itemsep=-2pt]
  \item \eqref{eq:clt} implies, in particular, that the solution $\sigma_t$ satisfies
    \begin{equation}\label{eq:cltWeak}\tag{\thesys-w}
      \br{\sigma_t}{\varphi}=\br{\sigma_0}{\varphi}+\int_0^t\!\br{\sigma_s}{J_s\varphi}\,ds+Z_t(\varphi)
    \end{equation}
    simultaneously for every $\varphi\in\ccz$.
  \item Observe that for any $\varphi_1,\dotsc,\varphi_k\in\ccz$, the process
    $\cramped{Z^{\varphi_1,\dotsc,\varphi_k}_t=(Z_t(\varphi_1),\dotsc,Z_t(\varphi_k))}$ is a
    continuous $\rr^k$-valued centered martingale with deterministic quadratic
    covariations, so it can be represented as
    \[Z^{\varphi_1,\dotsc,\varphi_k}_t\stackrel{d}{=}\int_0^t\!\big([C_s]^{\varphi_1,\dotsc,\varphi_k}\big)^{1/2}\,
    dB_s,\] where $\cramped{[C_t]^{\varphi_1,\dotsc,\varphi_k}}$ is the $k\!\times\!k$
    matrix-valued process with entries given by
    $\cramped{[C_t^{\varphi_1,\dotsc,\varphi_k}]_{ij}}=\cramped{C^{\varphi_i,\varphi_j}_t}$, 
    $\cramped{([C_t]^{\varphi_1,\dotsc,\varphi_k})^{1/2}}$ is the square root of this
    matrix, and $B_t$ is a standard $k$-dimensional Brownian motion. Thus, writing
    $\cramped{\left<\sigma_t;\varphi_1,\dotsc,\varphi_k\right>=\big(\!\br{\sigma_t}{\varphi_1},
      \dotsc,\br{\sigma_t}{\varphi_k}\!\big)}$ we have
    \begin{equation}\label{eq:finDim}
      \left<\sigma_t;\varphi_1,\dotsc,\varphi_k\right>\stackrel{d}{=}
      \int_0^t\!\left<\sigma_t;J_s\varphi_1,\dotsc,J_s\varphi_k\right>ds
      +\int_0^t\!\big([C_s]^{\varphi_1,\dotsc,\varphi_k}\big)^{1/2}\,dB_s.
    \end{equation}
  \item The limiting fluctuations $\sigma_t$ have zero mass in the following sense:
    whenever $\uno{}\in\ccz$ and $\br{\sigma_0}{\uno{}}=0$, $\br{\sigma_t}{\uno{}}=0$ for
    all $t\in[0,T]$ almost surely. This follows from \eqref{eq:finDim} simply by observing
    that, in this case, $J_s\uno{}$ and $C_s^{\uno{},\uno{}}$ are both always zero.
  \end{enumerate}
\end{remark}

Before presenting concrete examples where the setting and assumptions of this section
hold, we present a general condition which allows to deduce that the assumptions
\eqref{eq:iniFluc} on the initial distributions $\nu^N_0$, $\nu_0$, and $\sigma^N_0$ hold
(namely, that $\nu^N_0$ is a product measure).

\begin{teo}\label{thm:iniProd}
  In the setting of Theorem \ref{thm:clt}, assume that $\nu^N_0$ is the product of
  $N$ copies of a fixed probability measure $\nu_0\in\cp$ (i.e., $\nu^N_0$ is chosen by
  picking the initial type of each agent independently according to $\nu_0$), and that
  $\cramped{\ee\!\left(\br{\nu_0}{\rho^2_4}\right)<\infty}$.
  Then $\nu^N_0$ converges in distribution in $\cp$ to $\nu_0$, $\sigma^N_0$ converges in
  distribution in $\chup$ to a centered Gaussian $\chup$-valued random variable
  $\sigma_0$, and all the assumptions in \eqref{eq:iniFluc} are satisfied.
\end{teo}
 
\subsection{Application to concrete type spaces}
\label{sec:applTypeSpaces}

In this part we will present conditions under which the assumptions of Theorem \ref{thm:clt}
are satisfied in the four cases discussed at the beginning of this section.

\subsubsection{Finite \texorpdfstring{$W$}{W}}

\addtocounter{teo}{1}

This is the easy case. The reason is that $\cbw$ can be identified with $\rr^{|W|}$, and
thus $\sigma^N_t$ can be regarded as an $\rr^{|W|}$-valued process, so most of the
technical issues disappear. In particular, Theorem \ref{thm:clt} can be proved in this
case by arguments very similar to those leading to Theorem \ref{thm:limit}.

In the abstract setting of Theorem \ref{thm:clt}, it is enough to choose $\rho_i\equiv1$
and $\chii=\cci=\rr^{|W|}\cong\ell^2(W)$ for the right indices $i\in\{0,1,2,3,4\}$ in each
case. \eqref{eq:normIneq} follows simply from the finite-dimensionality of $\rr^{|W|}$ and
the equivalence of all norms in finite dimensions and \eqref{assum:rho:h},
\eqref{assum:rho:c}, and Assumption \ref{assum:2clt} are satisfied trivially.

Theorem \ref{thm:clt} takes a simpler form in this case. Write $W=\{w_1,\dotsc,w_k\}$,
\[\sigma^N_i\!(t)=\sigma^N_t\!(\{w_i\}),\quad\qquad
f_i(\sigma)=\sum_{j=1}^kJ_s\uno{\{w_i\}}(w_j)\,\sigma_j,\quad
\text{and}\quad\qquad g_{ij}(t)=C^{\uno{\{w_i\}},\uno{\{w_j\}}}_t,\]
where $\sigma$ above is in $\rr^k$. Also write
$F(\sigma)=\big(f_1(\sigma),\dotsc,f_k(\sigma)\big)$ and
$G(t)=\big(g_{ij}(t)\big)_{i,j=1\dotsc,k}$. Observe that $G(t)$ is a positive semidefinite
matrix for all $t\geq0$.

\begin{teotype}\label{thm:type:finite}
  In the above context, assume that Assumption \ref{assum:1clt} holds and that
  \[\sqrt{N}\left(\nu^N_0-\nu_0\right)\Longrightarrow\sigma_0
  \qquad\text{and}\qquad
  \sup_{N>0}\ee\!\left(\left|\sqrt{N}\left(\nu^N_0-\nu_0\right)\right|^2\right)<\infty,\]
  where the probability measures $\nu^N_t$ and $\nu_t$ are taken here as elements of
  $[0,1]^k$ and $\sigma_0\in\rr^k$.  Then the sequence of processes $\sigma^N\!(t)$
  converges in distribution in $D([0,T],\rr^k)$ to the unique solution $\sigma(t)$ of the
  following system of stochastic differential equations:
  \begin{equation}\label{eq:finiteClt}\tag{\thesys-f}
    d\sigma(t)=F(\sigma(t))\,dt+G^{1/2}(t)\,dB_t,
  \end{equation}
  where $B_t$ is a standard $k$-dimensional Brownian motion.
\end{teotype}

\begin{ex}
  This example provides a very simple model of agents changing their opinions on some
  issue of common interest, with rates of change depending on the ``popularity'' of each
  alternative. These opinions will be represented by $W=\{-m,\dotsc,m\}$ ($m$ can be
  thought of as being the strongest agreement with some idea, $0$ as being neutral, and
  $-m$ as being the strongest disagreement with it). Alternatively, one could think of the
  model as describing the locations of the agents, who move according to the density of
  agents at each site.

  The agents move in two ways. First, each agent feels attracted to other positions
  proportionally to the fraction of agents occupying them. Concretely, we assume that an
  agent at position $i$ goes to position $j$ at rate $\beta q_{i,j}\nu^N_t(\{j\})$, where
  $Q=(q_{i,j})_{i,j\in W}$ is the transition matrix of a Markov chain on
  $W$. One interpretation of these rates is that each agent decides to try to
  change its position at rate $\beta$, chooses a possible new position $j$ according to
  $Q$, and then changes its position with probability $\nu_t(\{j\})$ and stays put with
  probability $1-\nu_t(\{j\})$.  Second, each agent leaves its position at a rate
  proportional to the fraction of agents at its own position. We assume then that, in
  addition to the previous rates, each agent at position $i$ goes to position $j$ at rate
  $\alpha p_{i,j}\nu^N_t(\{i\})$, where $P=(p_{i,j})_{i,j\in W}$ is defined analogously to
  $Q$. This can be thought of as the agent leaving its position $i$ due to
  ``overcrowding'' at rate $\alpha\nu_t(\{i\})$ and choosing a new position according to
  $P$. We assume for simplicity that $p_{i,i}=q_{i,i}=0$ for all $i\in W$.

  We will set up the rates using the notation of Assumption \ref{assum:1clt}:
  \[\Gamma(i,k,\{j\})=
  \begin{cases}
    \alpha p_{i,j} & \text{if $k=i$}\\
    \beta q_{i,j} & \text{if $k=j$}\\
    0 & \text{otherwise}
  \end{cases}\qquad\text{and}\qquad
  \Lambda\equiv0.\]

  Assume that $\nu^N_0$ converges in distribution to some $\nu_0\in\cp$, let $\nu_t$ be
  the limit given by Theorem \ref{thm:limit} and write $u_t(i)=\nu_t(\{i\})$. It is easy
  to check that $u_t$ satisfies
  \[\frac{du_t(i)}{dt}=\alpha\sum_{j=-m}^mp_{j,i}u_t(j)^2-\alpha u_t(i)^2
  +\beta\sum_{j=-m}^m\big[q_{j,i}-q_{i,j}\big]u_t(i)u_t(j).\]

  Now let $\sigma_t$ be the limit in distribution of the fluctuations process
  $\sqrt{N}\big(u^N_t-u_t\big)$, and assume that the initial distributions $\nu^N_0$ and
  $\nu_0$ satisfy the assumptions of Theorem \ref{thm:type:finite}. It easy to check as
  before that
  \begin{equation}
    F_i(\sigma_t)=2\alpha\sum_{j=-m}^mp_{j,i}u_t(j)\sigma_t(j)-2\alpha u_t(i)\sigma_t(i)
    +\beta\sum_{j=-m}^m\big[q_{j,i}-q_{i,j}\big]\big(u_t(i)\sigma_t(j)+u_t(j)\sigma_t(i)\big).
  \end{equation}
  Thus, after computing the quadratic covariations we obtain the following: if $\star$
  denotes the coordinate-wise product in $\rr^{|W|}$ (i.e., $u\star v(i)=u(i)v(i)$) then
  the limiting fluctuations process $\sigma_t$ solves
  \begin{multline*}
    \hspace{-.1in}d\sigma_t=2\alpha P^\mathrm{t}(u_t\!\star\!\sigma_t)\,dt-2\alpha u_t\!\star\!\sigma_t\,dt
    +\beta\Big(\big[Q^\mathrm{t}-Q\big]\sigma_t\Big)\!\star\!u_t\,dt\\
    +\beta\Big(\big[Q^\mathrm{t}-Q\big]u_t\Big)\!\star\!\sigma_t\,dt +\sqrt{G(t)}\,dB_t,
  \end{multline*}
  where $B_t$ is a $(2m+1)$-dimensional standard Brownian motion and $G(t)$ is
  given by
  \[G_{i,j}(t)=\begin{dcases*}
    \alpha\sum_{k\neq i}p_{k,i}u_t(k)^2+\alpha u_t(i)^2 +\beta\sum_{k\neq
      i}\big(q_{k,i}+q_{i,k}\big)u_t(i)u_t(k) & if $i=j$\\
    -\alpha\big(p_{j,i}u_t(j)^2+p_{i,j}u_t(i)^2\big)-\beta\big(q_{j,i}+q_{i,j}\big)u_t(i)u_t(j)
    & if $i\neq j$.
  \end{dcases*}\]
\end{ex}

\subsubsection{\texorpdfstring{$W=\zz^d$}{W=Z\textcircumflex d}}\label{sec:typeZzd}

In this case $\cbw$ is no longer finite-dimensional and, moreover, the type space is not
compact, so we will need to make use of the weighting functions $\rho_i$. We let
$D=\lfloor d/2\rfloor+1$ and take
\[\rho_i(x)=\sqrt{1+|x|^{2iD}}.\]
Clearly, we have in this case that $\rho_1^p\leq C\rho_4$ for $C=p=2$.

Consider the following spaces:
\begin{gather}
  \ccz=\ell^\infty(\zz^d)= \left\{\varphi\!:\zz^d\rightarrow\rr\,\text{ such that
    }\norm\varphi_\infty<\infty\right\},\\
  \cci=\ell^{\infty,iD}(\zz^d)= \left\{\varphi\!:\zz^d\rightarrow\rr\,\text{ such that
    }\norm\varphi_{\infty,iD}=\sup_{x\in\zz^d}\frac{\left|\varphi(x)\right|}{1+|x|^{iD}}<\infty\right\}
  \quad\text{($i=2,3$)},\\
  \chii=\ell^{2,iD}(\zz^d)=\bigg\{\varphi\!:\zz^d\rightarrow\rr\,\text{ such that }
    \norm\varphi^2_{2,iD}=\sum_{x\in\zz^d}\frac{|\varphi(x)|^2}{1+|x|^{2iD}}<\infty\bigg\}
  \quad\text{($i=1,2,3,4$)},
\end{gather}
endowed with the norms defined within these definitions (we observe that $\rho_i$ does
not appear explicitly in the definition of the spaces $\cci$, but the definition does not
change if we replace the weighting function $1+|x|^{iD}$ appearing there by
$\rho_i$). These spaces are easily checked to be Banach (the $\cci$) and Hilbert
(the $\chii$) as required. With these definitions we have the following continuous
embeddings:
\begin{equation}\label{eq:embZzd}
  \begin{split}
    \ell^\infty(\zz^d)\xhookrightarrow[\quad]{}\ell^{2,D}(\zz^d)\xhookrightarrow[\text{c}]{\quad}
    \ell^{2,2D}(\zz^d)&\xhookrightarrow[\quad]{}\ell^{\infty,2D}(\zz^d)\xhookrightarrow[\quad]{}
    \ell^{2,3D}(\zz^d)\\
    &\xhookrightarrow[\quad]{}\ell^{\infty,3D}(\zz^d)\xhookrightarrow[\quad]{}
    \ell^{2,4D}(\zz^d)
  \end{split}
\end{equation}
(these embeddings will be proved in the proof of Theorem \ref{thm:type:zzd}).

To obtain \eqref{assum:2clt:1} and \eqref{assum:2clt:2} we will need to assume now that
\begin{subequations}
  \begin{gather}
    \sum_{y\in\zz^d}|y|^{8D}\Gamma(x,z,\{y\})\leq
    C\left(1+|x|^{8D}+|z|^{8D}\right)\qquad\text{and}\label{eq:mom:zzd:1}\\
    \qquad\sum_{y_1,y_2\in\zz^d}\left(|y_1|^{8D}+|y_2|^{8D}\right)\Lambda(x_1,x_2,z,\{(y_1,y_2)\})\leq
    C\left(1+|x_1|^{8D}+|x_2|^{8D}+|z|^{8D}\right)\label{eq:mom:zzd:2}
  \end{gather}
\end{subequations}
for all $x_1,x_2,z\in\zz^d$ (the other six inequalities in \eqref{assum:2clt:1} and
\eqref{assum:2clt:2} follow from these two and Jensen's inequality). We remark that in
\citet{meleard} the author also needs to assume moments of order $8D$ for the jump rates
($8D+2$ in her case, see $(\text{H}'_1)$ in her paper).

\begin{teotype}\label{thm:type:zzd}
  In the above context, suppose that Assumption \ref{assum:1clt} holds
  and that \eqref{eq:iniFluc}, \eqref{eq:mom:zzd:1}, and \eqref{eq:mom:zzd:2} hold. Then
  the conclusion of Theorem \ref{thm:clt} is valid, i.e., $\sigma^N_t$ converges in
  distribution in $D([0,T],\ell^{-2,D}(\zz^d))$ (where $\ell^{-2,D}(\zz^d)$ is the dual of
  $\ell^{2,D}(\zz^d)$) to the unique solution $\sigma_t$ of the
  ($\ell^\infty(\zz^d)^\prime$-valued) system given in \eqref{eq:clt}.
\end{teotype}

We recall that the dual of $\ell^\infty(\zz^d)$ can be identified with the space of
finitely additive measures on $\zz^d$, and thus every $\xi\in\ell^\infty(\zz^d)^\prime$
can be represented as $\big(\xi(x)\big)_{x\in\zz^d}$ and we can write
\[\br{\xi}{\varphi}=\sum_{x\in\zz^d}\varphi(x)\xi(x)\]
for $\varphi\in\ell^\infty(\zz^d)$. Therefore, \eqref{eq:clt} can be expressed in this
case in a manner analogous to \eqref{eq:finiteClt}.

\begin{ex}\label{ex:F-V}
  Here we consider a well-known model in mathematical biology, the Fleming-Viot process,
  which was originally introduced in \citet{flViot} as a stochastic model in population
  genetics with a constant number of individuals which keeps track of the positions of the
  individuals. We will actually consider the version of this model studied in
  \citet{ferMar}.

  We take as a type space $W=\zz^+$ and consider an infinite matrix $Q=(q(i,j))_{i,j\in
    W\cup\{0\}}$ corresponding to the transition rates of a conservative continuous-time
  Markov process on $W\cup\{0\}$, for which 0 is an absorbing state (observe that, in
  particular, $\cramped{q(i,i)=-\sum_{j\neq i}q(i,j)}$). Each individual moves
  independently according to $Q$, until it gets absorbed at $0$. On absorption, it chooses
  an individual uniformly from the population and jumps (instantaneously) to its position.
  We assume that the exit rates from each site are uniformly bounded, i.e.,
  $\cramped{\sup_{i\geq1}\sum_{j\in(W\cup\{0\})\setminus\{i\}}q(i,j)<\infty}$ (this is so
  that \eqref{assum:1:bd} is satisfied). The rates take the following form:
  \[\Gamma(i,k,\{j\})=
  \begin{cases}
    q(i,j) & \text{if $k\neq j$ and $i\neq j$}\\
    q(i,j)+q(i,0) & \text{if $k=j$ and $i\neq j$}\\
    0 & \text{if $i=j$}
  \end{cases}\qquad\text{and}\qquad
  \Lambda\equiv0.\] Observe that with this definition, the total rate at which a particle
  at $i$ jumps to $j$ when the whole population is at state $\nu$ is given by
  $q(i,j)+q(i,0)\nu(\{j\})$.

  We will write $u^N_t(i)=\nu^N_t(\{i\})$. It is clear that this model satisfies the
  assumptions of Theorem \ref{thm:limit}. Therefore, if the initial distributions $u^N_0$
  converge, and we denote by $u_t$ the limit given by Theorem \ref{thm:limit}, we obtain
  that for each $i\geq1$,
  \[\frac{du_t(\{i\})}{dt}=\sum_{j\geq1}\big[q(i,j)+q(i,0)u_t(j)\big]u_t(i).\]
  This limit was obtained in Theorem 1.2 of \citet{ferMar} (though there the convergence
  is proved for each fixed $t$).

  To study the fluctuations process we need to add the following moment assumption on $Q$:
  \[\sum_{j\geq1}j^8q(i,j)\leq C(1+i^8)\]
  for some $C>0$ independent of $i$. With this, if \eqref{eq:iniFluc} holds, we can apply
  Theorem \ref{thm:clt}. By the remark following Theorem \ref{thm:type:zzd}, to express
  the limiting system for the fluctuations process it is enough to apply
  \eqref{eq:cltWeak} to functions of the form $\varphi=\uno{i}$ for each $i\geq1$. Doing
  this, and after some algebraic manipulations, we deduce that the limiting fluctuations
  process $\sigma_t$ is the unique process with paths in
  $C([0,T],\ell^{\infty}(\zz^+)^\prime)$ satisfying the following stochastic differential
  equation:
  \[d\sigma_t=Q^\mathrm{t}\sigma_t\,dt+\left(\sum_{k\geq1}Q(k,0)\sigma_t(k)\right)\!u_t\,dt
  +\left(\sum_{k\geq1}Q(k,0)u_t(k)\right)\!\sigma_t\,dt+\sqrt{V_t}\,dB_t,\] where $B_t$ is
  an infinite vector of independent standard Brownian motions and $V_t$ is given by
  \[V_t(i,j)=\begin{dcases*}
    \sum_{k\neq
      i}\big[q(k,i)+q(k,0)u_t(i)\big]u_t(k)-\big[q(i,i)-q(i,0)\big]u_t(i)+q(i,0)u_t(i)^2
    & if $i=j$,\\
    -q(i,j)u_t(i)-q(j,i)u_t(j)-\big[q(i,0)+q(j,0)\big]u_t(i)u_t(j) & if $i\neq j$.
  \end{dcases*}\]
\end{ex}

\subsubsection{\texorpdfstring{$W=\Omega$}{W=\textbackslash Omega}, a compact,
  sufficiently smooth subset of \texorpdfstring{$\rr^d$}{R\textcircumflex d}}

Unlike the last case, the type space $W$ is now compact, so we can simply take
$\rho_i\equiv1$. Nevertheless, $W$ is not a discrete set now, and this leads us to use
Sobolev spaces for our sequence of continuous embeddings:
\begin{equation}\label{eq:embOmega}
  \cc^{3D}(\Omega)\xhookrightarrow[\quad]{}H^{3D}(\Omega)\xhookrightarrow[\text{c}]{\quad}
  H^{2D}(\Omega)\xhookrightarrow[\quad]{}\cc^D(\Omega)\xhookrightarrow[\quad]{}
  H^D(\Omega)\xhookrightarrow[\quad]{}\cc(\Omega)\xhookrightarrow[\quad]{}
  L^2(\Omega)
\end{equation}
(with $D=\lfloor d/2\rfloor+1$ as before), where $\cc^{k}(\Omega)$ is the space
of continuous functions on $\Omega$ with $k$ continuous derivatives, endowed with the norm
\[\norm\varphi_{\cc^k(\Omega)}=\sum_{|\alpha|\leq
  k}\sup_{x\in\Omega}\left|\partial^\alpha\varphi(x)\right|,\] and $H^k(\Omega)$ is the
Sobolev space (with respect to the $L^2(\Omega)$ norm) of order $k$, i.e., the space of
functions on $\Omega$ with $k$ weak derivatives in $L^2(\Omega)$, endowed with the norm
\[\norm\varphi^2_{H^{k}(\Omega)}=\sum_{|\alpha|\leq
  k}\int_\Omega\!\left|\partial^\alpha\varphi(x)\right|^2dx.\] The above embeddings are
either direct or are consequences of the usual Sobolev embedding theorems, see Theorem
4.12 of \citet{adams}. For these to
hold we need $\Omega$ to be sufficiently smooth (a locally Lipschitz boundary is
enough). The compact embedding $H^{2D}(\Omega)\hookrightarrow H^D(\Omega)$ is a
consequence of the Rellich--Kondrakov Theorem (see Theorem 6.3 of \citet{adams}).

In this case \eqref{assum:2clt:1} and \eqref{assum:2clt:2} hold trivially.
\eqref{assum:2clt:3} is much more delicate, and we will just leave it stated as it
is. (The assumptions (H3), (H3)${}^\prime$, and (H3)$^{\prime\prime}$ of \citet{meleard}
give some particular conditions which, if translated to our setting, would assure that
\eqref{assum:2clt:3} holds. These conditions are suitable in her setting but they
unfortunately rule out some interesting examples for us).

\begin{teotype}\label{thm:type:omega}
  In the above context, assume that Assumption and \ref{assum:1clt} holds,
  and that \eqref{assum:2clt:3} and \eqref{eq:iniFluc} hold. Then the conclusion of
  Theorem \ref{thm:clt} is valid, i.e., $\sigma^N_t$ converges in distribution in
  $D([0,T],H^{-3D}(\Omega))$ (where $H^{-3D}(\Omega)$ is the dual of $H^{3D}(\Omega)$) to
  the unique solution $\sigma_t$ of the ($\cc^{3D}(\Omega)^\prime$-valued) system given in
  \eqref{eq:clt}.
\end{teotype}

\subsubsection{\texorpdfstring{$W=\rr^d$}{W=R\textcircumflex d}}\label{sec:type:rrd}

This case combines both of the difficulties encountered before: $W$ is neither discrete
nor compact. To get around these problems we need to use now weighted Sobolev spaces. The
weighting functions $\rho_i$ are given by
\[\rho_i(x)=\sqrt{1+|x|^{2iD+2q}},\] where $D=\lfloor d/2\rfloor+1$ and $q\in\nn$ (to be chosen).
We consider now the spaces $\cc^{j,k}$ of continuous functions $\varphi$ with continuous
partial derivatives up to order $j$ and such that
$\lim_{|x|\rightarrow\infty}\left|\partial^\alpha\varphi(x)\right|/(1+|x|^k)=0$ for all
$|\alpha|\leq j$, with the norms
\[\norm\varphi_{\cc^{j,k}}=\sum_{|\alpha|\leq
  j}\sup_{x\in\rr^d}\frac{\left|\partial^\alpha\!\varphi(x)\right|} {1+|x|^k},\] (as in
Section \ref{sec:typeZzd}, the weigthing functions $\rho_i$ do not appear explicitly here,
but the definition does not change if we replace the term $1+|x|^k$ by
$\sqrt{1+|x|^{2k}}$) and the weighted Sobolev spaces $W^{j,k}_0$ (with respect to the
$L^2$ norm) defined as follows: we define the norms
\[\norm\varphi^2_{W^{j,k}_0}=\sum_{|\alpha|\leq
  j}\int_{\rr^d}\frac{\left|\partial^\alpha\!\varphi(x)\right|^2} {1+|x|^{2k}}\,dx\]
and
let $W^{j,k}_0$ be the closure in $L^2$ under this norm of the space of functions of class
$\cc^\infty$ with compact support.

The right sequence of embeddings is now the following:
\begin{equation}\label{eq:embRrd}
  \cc^{3D,q}\xhookrightarrow[\quad]{}W^{3D,D+q}_0\xhookrightarrow[\text{c}]{\quad}
  W^{2D,2D+q}_0\xhookrightarrow[\quad]{}C^{D,2D+q}\xhookrightarrow[\quad]{}
  W^{D,3D+q}_0
  \xhookrightarrow[\quad]{}\cc^{0,3D+q}\xhookrightarrow[\quad]{}W^{0,4D+q}_0.
\end{equation}
$q\in\nn$ can be chosen depending on the specific example being analyzed: $q=0$ works for
many examples, but as we will see in the next example, choosing a positive $q$ ($q=1$ in
that case) can help, for instance, by making all constant functions be in
$\cc^{3D,q}$. These embeddings are, as before, either straightforward or consequences of
the usual Sobolev embedding theorems (adapted now to the weighted case; see
\citet{meleard}, where the author uses the same type of embeddings, and see \citet{kufner} for
a general discussion of weighted Sobolev spaces).

To obtain \eqref{assum:2clt:1} and \eqref{assum:2clt:2} we need to add the following
moment assumptions on the rates, analogous to those we used in Theorem \ref{thm:type:zzd}:
for all $x,x_1,x_2,z\in\rr^d$,
\begin{subequations}
  \begin{gather}
    \int_{\rr^d}\!|y|^{8D+2q}\,\Gamma(x,z,dy)\leq C\left(1+|x|^{8D+2q}+|z|^{8D+2q}\right)
    \qquad\text{and}\label{eq:mom:rrd:1}\\
  \begin{multlined}
    \int_{\rr^d\!\times\!\rr^d}\!\left(|y_1|^{8D+2q}+|y_2|^{8D+2q}\right)\,\Lambda(x_1,x_2,z,dy_1\ootimes
    dy_2)\\
    \hspace{1.4in}\leq C\left(1+|x_1|^{8D+2q}+|x_2|^{8D+2q}+|z|^{8D+2q}\right).\label{eq:mom:rrd:2}
  \end{multlined}
\end{gather}
\end{subequations}
We observe that the power $8D+2q$ appearing in this assumption corresponds exactly, when
$q=1$, to the moments of order $8D+2$ assumed in $(\text{H}'_1)$ in \citet{meleard}.
\eqref{assum:2clt:3}, as in the previous case, is much more involved, so we will again
leave it stated as it is.

\begin{teotype}\label{thm:type:rrd}
  In the above context, assume moreover that Assumption \ref{assum:1clt} holds, and that
  \eqref{eq:iniFluc}, \eqref{assum:2clt:3}, \eqref{eq:mom:rrd:1}, and \eqref{eq:mom:rrd:2}
  hold. Then the conclusion of Theorem \ref{thm:clt} is valid, i.e., $\sigma^N_t$
  converges in distribution in $D([0,T],W^{-3D,D+q}_0)$ (where $W^{-3D,D+q}_0$ is the dual
  of $W^{3D,D+q}_0$) to the unique solution $\sigma_t$ of the
  ($(\cc^{3D,q})^\prime$-valued) system given in \eqref{eq:clt}.
\end{teotype}

\begin{ex}[Continuation of Example \ref{ex:infPerc}]
  In the previous section we obtained the system \eqref{eq:infPerc} that characterizes the
  information percolation model of \citet{dufMan} by using \eqref{eq:diffSys}. If we use
  \eqref{eq:intDiffSys} instead we obtain
  \[\frac{d}{dt}\br{\nu_t}{\varphi}
  =2\lambda\br{\nu_t}{\nu_t\!\ast\!\varphi}
  -2\lambda\br{\nu_t}{\varphi}\] for all $\varphi\in{\cal
    B}(\rr)$, where
  $(\nu_s\!\ast\!\varphi)(z)=\int_W\!\varphi(x+z)\,\nu_s(dx)$.
  
  To obtain the fluctuations limit, we need to check the assumptions of Theorem
  \ref{thm:type:rrd}. As we mentioned, we will take $q=1$. Assumption \ref{assum:1clt}
  holds trivially because $\lambda(w_1,w_2,\nu)$ and $b(w_1,w_2,\nu,\cdot)$ do not depend
  on $\nu$. We will assume that the initial distribution of the system satisfies
  \eqref{eq:iniFluc}. \eqref{eq:mom:rrd:1} and \eqref{eq:mom:rrd:2} are straightforward to
  check in this case.

  We are left checking \eqref{assum:2clt:3}. Let $\varphi\in\cc^{3,1}$ and take
  $\mu_1,\mu_2,\mu_3,\mu_4\in\cp$ having moments of order 10. We have that
  \begin{multline}
    J_{\mu_1,\mu_2}\varphi(z)=\int_{-\infty}^\infty\!\int_{-\infty}^\infty\!
    \big(2\varphi(w_1+w_2)-\varphi(w_1)-\varphi(w_2)\big)\,\mu_1(dw_2)\,\mu_1(dw_1)\\
    +\int_{-\infty}^\infty\!\big(2\varphi(w+z)-\varphi(w)-\varphi(z)\big)\,
    \big[\mu_1(dw)+\mu_2(dw)\big].
  \end{multline}
  The first term on the right side is constant in $z$, so it is in $\cc^{3,1}$ (this is
  why we needed $q=1$ in this example). For the second term, since $|\varphi(x)|\leq
  C\norm\varphi_{\cc^{3,1}}(1+|x|)$ and $\br{\mu_i}{1+|\cdot|^{10}}<\infty$, the integral is
  bounded, and hence the derivatives with respect to $z$ can be taken inside the integral,
  whence we get that this term is also in $\cc^{3,1}$. The same argument can be
  repeated for $\cc^{1,3}$ and $\cc^{0,4}$. The fact that the norm of this operator in
  these spaces is bounded uniformly in $\mu_1,\mu_2$ follows from the same
  argument. This gives \eqref{assum:2clt:3i}. Using the same formula it is easy to
  show that
  \[\norm{\big(J_{\mu_1,\mu_2}-J_{\mu_3,\mu_4}\big)\varphi}_{\cc^{3,1}}
  \leq C\norm\varphi_{\cc^{3,1}}\!\Big[\norm{\mu_1-\mu_3}_{(\cc^{3,1})^\prime}
  +\norm{\mu_2-\mu_4}_{(\cc^{3,1})^\prime}\Big],\] which is stronger than
  \eqref{assum:2clt:3ii}.

  We have checked all the assumptions of Theorem \ref{thm:type:rrd}, so we deduce that the
  fluctuations process $\sigma^N_t$ converges in distribution in $W^{-3,2}_0$ to the unique
  solution of \eqref{eq:clt} (which is an equation in $(\cc^{3,1})^\prime$). Writing down
  the formula for $J_s$ in this case yields
  \[\br{\sigma_s}{J_s\varphi}=4\lambda\br{\sigma_s}{\nu_s\!\ast\!\varphi}-2\lambda\br{\sigma_s}{\varphi}\]
  for every $\varphi\in\cc^{3,1}$. For the quadratic covariations we get
  \[C^{\varphi_1,\varphi_2}_s =4\lambda\br{\nu_s}{\nu_s\!\ast\!(\varphi_1\varphi_2)}
  -6\lambda\br{\nu_s}{\varphi_1}\br{\nu_s}{\varphi_2}
  +2\lambda\br{\nu_s}{\varphi_1\varphi_2}\] for every $\varphi_1,\varphi_2\in\cc^{3,1}$.
  Therefore the limiting fluctuations satisfy
  \[\br{\sigma_t}{\varphi}
  =\br{\sigma_0}{\varphi}
  +\lambda\int_0^t\!\left[4\br{\sigma_s}{\nu_s\!\ast\!\varphi}-2\br{\sigma_s}{\varphi}\right]ds
  +Z_t(\varphi),\] with $Z_t$ being a centered Gaussian process taking values in the dual
  of $\cc^{3,1}$ with quadratic covariations given by $[Z(\varphi_1),Z(\varphi_2)]_t=\int_0^tC^{\varphi_1,\varphi_2}_s\,ds$
  for each $\varphi_1,\varphi_2\in\cc^{3,1}$.
\end{ex}

\section{Proofs of the Results}\label{sec:proofs}

Throughout this section, $C$, $C_1$, and $C_2$ will denote constants whose values might
change from line to line.

\subsection{Preliminary computations and proof of Theorem \ref{thm:limit}}

Since $\nu^N_t$ is a jump process in $\cp$ with bounded jump rates, its generator is given
by
\begin{equation}\label{eq:genN}
  \begin{aligned}
    \Omega_N\!f(\nu)&=N\int_W\!\gamma(w,\nu)\int_W\!\Delta_Nf(\nu;w;w')\,a(w,\nu,dw')\,\nu(dw)\\
    &\hspace{0.3in}+N\int_W\!\int_W\!\lambda(w_1,w_2,\nu)\int_{W\!\times
      W}\Delta_Nf(\nu;w_1,w_2;w_1'w_2')\\
    &\hspace{2in}\cdot b(w_1,w_2,\nu,dw_1\ootimes dw_2'))\,\nu(dw_1)\,\nu(dw_2)
  \end{aligned}
\end{equation}
for any bounded measurable function $f$ on $\cp$, where
$\Delta_Nf(\nu;w;w')=f\big(\nu+N^{-1}(\delta_{w'}-\delta_w)\big)-f(\nu)$ and
$\Delta_Nf(\nu;w_1,w_2;w_1',w_2')=f\big(\nu+N^{-1}(\delta_{w_1'}+\delta_{w_2'}-\delta_{w_1}-\delta_{w_2})\big)-f(\nu)$.

Given $\varphi\in\mw$ we get by using \eqref{eq:genN} and Proposition IV.1.7 of \cite{ethKur} for
$f(\nu)=\br{\nu}{\varphi}$ that
\begin{equation}
  \begin{split}
    \br{\nu^N_t}{\varphi}&=\br{\nu^N_0}{\varphi}+M^{N,\varphi}_t
    +\int_0^t\!\int_W\!\gamma(w,\nu^N_s)\int_W\!\dif\,a(w,\nu^N_s,dw')\,\nu^N_s(dw)\,ds\\
    &\hspace{0.35in}+\int_0^t\!\int_W\!\int_W\!\lambda(w_1,w_2,\nu^N_s)\int_{W\times
      W}\!\difd\\
    &\hspace{2in}\cdot
    b(w_1,w_2,\nu^N_s,dw_1'\ootimes dw_2)\,\nu^N_s(dw_2)\,\nu^N_s(dw_1)\,ds,\\
  \end{split}\label{eq:intForm}
\end{equation}
where $M^{N,\varphi}_t$ is a martingale starting at 0.  This formula is the key to the
proof of Theorem \ref{thm:limit} because, ignoring the martingale term, this equation has
the exact form we need for obtaining \eqref{eq:intDiffSys}, and thus the idea will be to show that
$M^{N,\varphi}_t$ vanishes in the limit as $N\to\infty$. This follows from the fact that
the quadratic variation of $M^{N,\varphi}_t$ is of order $O(1/N)$. More precisely, we have
the following formula: for any $\varphi_1,\varphi_2\in\mw$, the predictable quadratic
covariation between the martingales $M^{N,\varphi_1}_t$ and $M^{N,\varphi_2}_t$ is given
by
  \begin{equation}\label{eq:formCoQuadN}
    \begin{aligned}
      \Big<M^{N,\varphi_1}&,M^{N,\varphi_2}\Big>_t
      =\frac{1}{N}\int_0^t\!\int_W\!\gamma(w,\nu^N_s)\int_W\!(\varphi_1(w')-\varphi_1(w))(\varphi_2(w')-\varphi_2(w))\\
      &\hspace{2.8in}\cdot a(w,\nu^N_s,dw')\,\nu^N_s(dw)\,ds\\
      &\hspace{0in}+\!\frac{1}{N}\int_0^t\int_W\!\int_W\!\lambda(w_1,w_2,\nu^N_s)\int_{W\!\times\!W}\!
      (\varphi_1(w_1')+\varphi_1(w_2')-\varphi_1(w_1)-\varphi_1(w_2))\\
      &\hspace{2.1in}\cdot(\varphi_2(w_1')+\varphi_2(w_2')-\varphi_2(w_1)-\varphi_2(w_2))\\
      &\hspace{1.7in}\cdot b(w_1,w_2,\nu^N_s,dw_1'\ootimes
      dw_2')\,\nu^N_s(dw_2)\,\nu^N_s(dw_1)\,ds.
    \end{aligned}
  \end{equation}
  The proof of this formula is almost the same as that of Proposition 3.4 of \citet{fourMel}
  so we will omit it (there the computation is done for $\varphi_1=\varphi_2$, but the
  generalization is straightforward, and can also be obtained by polarization).
  
\begin{proof}[Proof of Theorem \ref{thm:limit}]
  The proof is relatively standard, and its basic idea is the following. First one proves that
  the sequence of processes $\br{\nu^N_t}{\varphi}$ is tight in $D([0,T],\rr)$ for each
  $\varphi\in\cbw$, which in turn implies the tightness of $\nu^N_t$ in $D([0,T],\cp)$.
  The tightness of these processes follows from standard techniques and uses
  \eqref{eq:intForm} and \eqref{eq:formCoQuadN}. Next, one uses a martingale argument and
  \eqref{eq:formCoQuadN} to show that any limit point of $\nu^N_t$ satisfies
  \eqref{eq:intDiffSys}. Finally, using Gronwall's Lemma one deduces that
  \eqref{eq:intDiffSys} has a unique solution. We refer the reader to the proof of Theorem
  5.3 of \citet{fourMel} for the details.
\end{proof}

\begin{proof}[Proof of Corollary \ref{cor:absCont}]
  Denote by $(\tau^N_i)_{i>0}$ the sequence of stopping times corresponding to the jumps
  of the process $\nu^N_t$. Let $A$ be any Borel subset of $W$ with $\mu(A)=0$ and let
  $\varphi$ be any positive function in $\mw$ whose support is contained in $A$. By
  \eqref{eq:intForm}, for every $t\in[0,T]$ we have that
  \begin{equation}\label{eq:expBr}
    \begin{aligned}
      \ee&\left(\br{\nu^N_{t\wedge\tau_1^N}}{\varphi}\right)=\ee\left(\br{\nu_0}{\varphi}\right)+\ee\left(M^{N,\varphi}_{t\wedge\tau_1^N}\right)\\
      &\hspace{0.3in}+\ee\left(\int_0^{t\wedge\tau_1^N}\!\int_W\!\gamma(w,\nu^N_s)\int_W\!\dif\,a(w,\nu^N_s,dw')\,\nu^N_s(dw)\,ds\right)\\
      &\hspace{0.3in}+\ee\Bigg(\int_0^{t\wedge\tau_1^N}\!\int_W\!\int_W\!\lambda(w_1,w_2,\nu^N_s)\int_{W\!\times\!W}\!\difd\\
      &\hspace{1.9in}\cdot b(w_1,w_2,\nu^N_s,dw_1'\ootimes
      dw_2'))\,\nu^N_s(dw_2)\,\nu^N_s(dw_1)\,ds\Bigg).
    \end{aligned}
  \end{equation}
  The first term on the right side of \eqref{eq:expBr} is 0 because the support of
  $\varphi$ is contained $A$ and $\nu_0(A)=0$. The second term is 0 by Doob's Optional Sampling
  Theorem. For the third term observe that for $s<\tau^N_1$, $\nu^N_s=\nu_0$, so
  \begin{multline*}
    \ee\!\left(\left|\int_0^{t\wedge\tau_1^N}\!\int_W\!\gamma(w,\nu^N_s)\int_W\!\dif
        \,a(w,\nu^N_s,dw')\,\nu^N_s(dw)\,ds\right|\right)\\
    \leq\ogamma\ee\!\left(\int_0^{t\wedge\tau_1^N}\!\int_W\!\int_W\!|\varphi(w')-\varphi(w)|
      \,a(w,\nu_0,dw')\,\nu_0(dw)\right)
  \end{multline*}
  which is 0 since $\varphi$ is supported inside $A$ and the measure
  $\int_W\!a(w',\nu_0,\cdot)\,\nu_0(dw')$ is absolutely continuous with respect to
  $\mu$. The fourth term is 0 by analogous reasons. We deduce that the expectation on the
  left side of \eqref{eq:expBr} is 0, and therefore, since $\varphi$ is positive,
  $\big\langle\nu^N_{t\wedge\tau_1^N},\varphi\big\rangle=0$ with probability 1. In particular,
  $\nu^N_{t\wedge\tau_1^N}$ is absolutely continuous with respect to $\mu$ for all
  $t\in[0,T]$ with probability 1.

  Using the strong Markov property we obtain inductively that
  $\big\langle\nu^N_{t\wedge\tau_i^N},\varphi\big\rangle=0$ almost surely for every $i>0$ and
  $t\in[0,T]$. Since the jump rates of the process are bounded, there are finitely many
  jumps before $T$ with probability 1, and we deduce that $\br{\nu^N_t}{\varphi}=0$ almost
  surely for all $t\in[0,T]$.  Now if $\nu_t$ is the limit in distribution of the sequence
  $\nu^N_t$ given by Theorem \ref{thm:limit} and $\varphi\in\cbw$,
  $\ee\!\left(\br{\nu^N_t}{\varphi}\right)\to\br{\nu_t}{\varphi}$ as $N\to\infty$,
  so $\br{\nu_t}{\varphi}=0$ for all $t\in[0,T]$ whenever $\varphi$ is supported inside
  $A$, and the result follows.
\end{proof}

\subsection{Proof of Theorem \ref{thm:clt}}\label{sec:proofClt}

We will assume throughout this part that all the assumptions of Theorem \ref{thm:clt}
hold. For simplicity, we will also assume that $\Gamma\equiv0$ (these terms are easier to
handle and are in fact a particular case of the ones corresponding to $\Lambda$).

Before getting started we recall that, by Parseval's identity, given any $A\in\chip$ and a
complete orthonormal basis $(\phi_k)_{k\geq0}$ of $\chii$,
\[\|A\|^2_\chip=\sum_{k\geq0}|A(\phi_k)|^2.\] We will use this fact several times below.

\subsubsection{Moment estimates for \texorpdfstring{$\nu^N_t$ and $\nu_t$}{\textbackslash
    nu\textcircumflex N\textunderscore t and \textbackslash
    nu\textunderscore t}}

Recall that we have assumed that
$\sup_{N>0}\ee\!\left(\br{\nu^N_0}{\rho^2_4}\right)<\infty$ and
$\ee\!\left(\br{\nu_0}{\rho^2_4}\right)<\infty$. We need to show that these moment
assumptions propagate to $\nu^N_t$ and $\nu_t$:

\begin{prop}\label{prop:bdMoments}
  The following properties hold:
  \begin{subequations}
    \begin{gather}
      \sup_{N>0}\ee\!\left(\sup_{t\in[0,T]}\br{\nu^N_t}{\rho^2_4}\right)<\infty,\quad\text{and}\label{eq:mom2}\\
      \sup_{t\in[0,T]}\br{\nu_t}{\rho^2_4}<\infty.\label{eq:mom1}
    \end{gather}
  \end{subequations}
\end{prop}

The proof of this result will rely on an explicit construction of the process in terms of
Poisson point measures. This is similar to what is done in Section 2.2 of \citet{fourMel}
(though we will need to use a more abstract approach because our type spaces are not
necessarily Euclidean), so we will only sketch the main ideas.

We fix $N>0$ and consider the following random objects, defined on a sufficiently large
probability space: a $\cp$-valued random variable $\nu^N_0$ (corresponding to the initial
distribution) and a Poisson point measure $Q(ds,di,dj,du,d\theta)$ on
$[0,T]\!\times\!I_N\!\times\!  I_N\!\times\![0,1]\!\times\![0,1]$ with intensity measure
$(\olambda/N)\,ds\,di\,dj\,du\,d\theta$. We also consider a Blackwell-Dubins
representation $\varrho$ of $\cp(W\!\times\!W)$ with respect to a uniform random variable
on $[0,1]$, i.e., a continuous function
$\varrho\!:\cp(W\!\times\!W)\times[0,1]\longrightarrow W\!\times\!W$ such that
$\varrho(\xi,\cdot)$ has distribution $\xi$ (with respect to the Lebesgue measure on
$[0,1]$) for all $\xi\in\cp(W\!\times\!W)$ and $\varrho(\cdot,u)$ is continuous for almost
every $u\in[0,1]$ (see \citet{blackDub} for the existence of such a function). This gives
us an abstract way to use a uniform random variable to pick the pairs of types to which
agents go after interacting. Finally, we introduce the following notation:
$\eta^i(\nu^N_t)$ will denote the $i$-th type, with respect to some fixed total order of
$W$, appearing in $\nu^N_t$ (we recall that, under the axiom of choice, any set can be
well-ordered, and hence totally ordered; moreover, this ordering can be taken to be
measurable because $W$, being a Polish space, is measurably isomorphic to $[0,1]$). With
this definition, choosing a type uniformly from $\nu^N_t$ is the same as choosing $i$
uniformly from $I_N$ and considering the type given by $\eta^i(\nu^N_t)$.  Our process can
be represented then as follows:
\begin{equation}
\begin{aligned}
  \nu^N_t=\nu^N_0&+\int_0^t\!\int_{I_N}\!\int_{I_N}\!\int_0^1\!\int_0^1\!
  \frac{1}{N}\Big[\delta_{\varrho^1(b(\eta^i(\nu^N_{s-}),\eta^j(\nu^N_{s-}),\nu^N_{s-},\cdot),u)}\\
  &\hspace{1.5in}+\delta_{\varrho^2(b(\eta^i(\nu^N_{s-}),\eta^j(\nu^N_{s-}),\nu^N_{s-},\cdot),u)}-\delta_{\eta^i(\nu^N_{s-})}
  -\delta_{\eta^j(\nu^N_{s-})}\Big]\\
  &\hspace{2in}\cdot\uno{\theta\leq\lambda(\eta^i(\nu^N_{s-}),\eta^j(\nu^N_{s-}),\nu^N_{s-})/\olambda}\,\,
  Q(ds,di,dj,du,d\theta),
\end{aligned}\label{eq:constPoisson}
\end{equation}
where $\varrho^1$ and $\varrho^2$ are the first and second components of $\varrho$ (see
Definition 2.5 in \cite{fourMel} for more details on this construction).

\begin{proof}[Proof of Proposition \ref{prop:bdMoments}]
  Since $\br{\nu^N_t}{\rho^2_4}=\br{\nu^N_0}{\rho^2_4} +\sum_{0\leq s\leq
    t}\!\left[\br{\nu^N_s-\nu^N_{s-}}{\rho^2_4}\right]$, it is easy to deduce from
  the last equation that
  \begin{equation}
    \begin{aligned}
      \br{\nu^N_t}{\rho^2_4}&=\br{\nu^N_0}{\rho^2_4}
      +\int_0^t\!\int_{I_N}\!\int_{I_N}\!\int_0^1\!\int_0^1\!
      \frac{1}{N}\Big[\rho^2_4(\varrho^1(b(\eta^i(\nu^N_{s-}),\eta^j(\nu^N_{s-}),\nu^N_{s-},\cdot),u))\\
      &\hspace{0.9in}+\rho^2_4(\varrho^2(b(\eta^i(\nu^N_{s-}),\eta^j(\nu^N_{s-}),\nu^N_{s-},\cdot),u))
      -\rho^2_4(\eta^i(\nu^N_{s-}))-\rho^2_4(\eta^j(\nu^N_{s-}))\Big]\\
      &\hspace{2.15in}\cdot\uno{\theta\leq\lambda(\eta^i(\nu^N_{s-}),\eta^j(\nu^N_{s-}),\nu^N_{s-})/\olambda}\,\,
      Q(ds,di,dj,du,d\theta).
    \end{aligned}
  \end{equation}
  Taking expectations and ignoring the (positive) terms being subtracted we obtain
  \begin{align}
    \ee\!\left(\sup_{t\in[0,T]}\br{\nu^N_t}{\rho^2_4}\right)&\leq\ee\!\left(\br{\nu^N_0}{\rho^2_4}\right)
    +\frac{1}{N^2}\int_0^T\!\ee\Bigg(\sum_{i=1}^N\sum_{j=1}^N
    \lambda(\eta^i(\nu^N_{s-}),\eta^j(\nu^N_{s-}),\nu^N_{s-})\\
    &\hspace{1.25in}\cdot\int_0^1\!\Big[\rho^2_4(\varrho^1(b(\eta^i(\nu^N_{s-}),\eta^j(\nu^N_{s-}),\nu^N_{s-},\cdot),u))\\
      &\hspace{1.4in}+\rho^2_4(\varrho^2(b(\eta^i(\nu^N_{s-}),\eta^j(\nu^N_{s-}),\nu^N_{s-},\cdot),u))
      \Big]\,du\Bigg)\,ds\\
    &\hspace{-0.7in}=\ee\!\left(\br{\nu^N_0}{\rho^2_4}\right)
    +\int_0^T\!\ee\Bigg(\int_W\!\int_W\!\int_W\!\int_{W\!\times\!W}\!
    \big[\rho^2_4(w_1')+\rho^2_4(w_2')\big]\,\Lambda(w_1,w_2,z,dw_1'\ootimes dw_2')\\
    &\hspace{2.45in}\cdot\nu^N_s(dz)\,\nu^N_s(dw_2)\,\nu^N_s(dw_1)\Bigg)\,ds\\
    &\hspace{-0.7in}\leq\ee\!\left(\br{\nu^N_0}{\rho^2_4}\right)+C\int_0^T\!\ee\!\Bigg(\int_W\!\int_W\!\int_W\!
    \big[\rho^2_4(w_1)+\rho^2_4(w_2)+\rho^2_4(z)\big]\\
    &\hspace{2.2in}\cdot\nu^N_s(dz)\,\nu^N_s(dw_2)\,\nu^N_s(dw_1)\Bigg)\,ds\\
    &\hspace{-0.7in}\leq\ee\!\left(\br{\nu^N_0}{\rho^2_4}\right)+C\int_0^T\!
    \ee\!\left(\sup_{s\in[0,t]}\br{\nu^N_s}{\rho^2_4}\right)ds,
  \end{align}
  where we used \eqref{assum:2clt:2} in the second inequality.  By hypothesis
  $\ee\!\left(\br{\nu^N_0}{\rho^2_4}\right)$ is bounded uniformly in $N$, so by Gronwall's
  Lemma we deduce that
  \begin{equation}
    \ee\!\left(\sup_{t\in[0,T]}\br{\nu^N_t}{\rho^2_4}\right)\leq C_1e^{C_2T},\label{eq:bdMom2}
  \end{equation}
  with $C_1$ and $C_2$ being independent of $N$, whence \eqref{eq:mom2} follows.

  To get \eqref{eq:mom1}, write $(\rho^2_4\wedge L)(w)=\rho^2_4(w)\wedge L$, and observe
  that, since $\rho^2_4\wedge L\in\cbw$, Theorem \ref{thm:limit} implies that
  $\lim_{N\rightarrow\infty}\ee\!\left(\sup_{t\in[0,T]}\br{\nu^N_t}{\rho^2_4\wedge
      L}\right) =\sup_{t\in[0,T]}\br{\nu_t}{\rho^2_4\wedge L}$, so by \eqref{eq:bdMom2},
  \[\sup_{t\in[0,T]}\br{\nu_t}{\rho^2_4\wedge L}\leq C_1e^{C_2T}.\]
  Using the Monotone Convergence Theorem it is easy to check that
  $\sup_{s\in[0,T]}\br{\nu_s}{\rho^2_4\wedge L}\to\sup_{s\in[0,T]}\br{\nu_s}{\rho^2_4}$ as
  $L\rightarrow\infty$, and thus \eqref{eq:mom1} follows.
\end{proof}

For most of this section we will continue ignoring the type-process $\eta^N_t$, working
instead with the empirical distribution process $\nu^N_t$ we are interested in.  However,
we will need to consider $\eta^N_t$ directly in Step 2 of the proof of Theorem
\ref{thm:clt}, and we will need to use a moment estimate similar to \eqref{eq:mom2} for
this process. Observe that statement of the theorem (and that of Theorem \ref{thm:limit})
makes no assumption on the distribution of $\eta^N_0$, but instead only deals with the
initial empirical distribution $\nu^N_0$. Therefore we are free to choose $\eta^N_0$ in
any way compatible with $\nu^N_0$. For convenience we can construct $\eta^N_0$ in the
following way: assuming $\nu^N_0$ takes a specific value $\overline{\nu}^N_0\in\cp_a$,
choose $\eta^N_0(1)$ uniformly from $\overline{\nu}^N_0$ and then inductively choose
$\eta^N_0(i)$ uniformly from the remaining $N-i+1$ individuals, i.e., from
$\cramped{\big[N\overline{\nu}^N_0-\delta_{\eta^N_0(1)}-\dotsm-\delta_{\eta^N_0(i-1)}\big]/(N-i+1)}$.
It is clear then that, with this choice, $\eta^N_0$ is exchangeable and
$\frac{1}{N}\sum_{i=1}^N\delta_{\eta^N_0(i)} =\overline{\nu}^N_0$ as required. Moreover,
given any $i\in I_N$,
$\ee\!\left(\rho^2_4(\eta^N_0(i))\right)=\ee\!\left(\br{\nu^N_0}{\rho^2_4}\right)$, and
thus the moment assumption that we made on $\nu^N_0$ can be rewritten as
$\sup_{N>0}\sup_{i\in I_N}\ee\!\left(\rho^2_4(\eta^N_0(i))\right)<\infty$ for all $i\in
I_N$. The proof of \eqref{eq:mom2} can then be adapted (by modifying slightly the explicit
construction we made of $\nu^N_t$ to deal with $\eta^N_t$) to obtain
\begin{equation}
  \label{eq:mom2EtaN}
  \sup_{N>0}\sup_{i\in I_N}\ee\!\left(\sup_{t\in[0,T]}\rho^2_4(\eta^N_t(i))\right)<\infty.
\end{equation}
(We remark that the proof of this estimate uses \eqref{eq:mom2} itself).

\subsubsection{Extension of \texorpdfstring{$\br{\nu^N_t}{\cdot}$}{<\textbackslash
    nu\textcircumflex N\textunderscore t,\textcdot>} and
  \texorpdfstring{$\br{\nu_t}{\cdot}$}{<\textbackslash nu\textunderscore t,\textcdot>} to
  \texorpdfstring{$\chfp$}{H\textunderscore4'}}\label{sec:prelim}

The $\cp$-valued process $\nu^N_t$ can be seen as a linear functional on $\mw$ via the
mapping $\varphi\longmapsto\br{\nu^N_t}{\varphi}$, and the same can be done for
$\nu_t$. However, since $\chf$ consists of measurable but not necessarily bounded
functions, the integrals $\br{\nu^N_t}{\varphi}$ and $\br{\nu_t}{\varphi}$ may
diverge. Our first task will be to show that these integrals are finite and, moreover,
that $\nu^N_t$ (and $\nu_t$) can be seen as taking values in $\chfp$ (and thus also in all
the other dual spaces we are considering). A consequence of this will be that $\sigma^N_t$
is well defined as an $\chfp$-valued process.

\begin{prop}\label{prop:ext}
  \mbox{} The mapping $\varphi\in\chf\mapsto\br{\nu^N_t}{\varphi}$ is in $\chfp$ almost
  surely for every $t\in[0,T]$ and $N>0$. Analogously, the mapping
  $\varphi\in\chf\mapsto\br{\nu_t}{\varphi}$ is in $\chfp$ for every $t\in[0,T]$.

  Furthermore, $\nu_t$ satisfies \eqref{eq:intDiffSys} for every $\varphi\in\chf$, while
  $\nu^N_t$ satisfies \eqref{eq:intForm} for every $\varphi\in\chf$ almost surely. In
  particular, given any $\varphi\in\chf$, $M^{N,\varphi}_t$ is a martingale starting at 0
  such that the predictable quadratic covariations
  $\big<M^{N,\varphi_1},M^{N,\varphi_2}\big>_t$ are the ones given by the formula in
  \eqref{eq:formCoQuadN} for all $\varphi_1,\varphi_2\in\chf$.
\end{prop}

\begin{proof}
  We are only going to prove the assertions for $\nu^N_t$, the ones for $\nu_t$ can be
  checked similarly (and more easily).
  
  The first claim follows directly from \eqref{assum:rho:h} and Proposition
  \ref{prop:bdMoments}: for $\varphi\in\chf$,
  \[\left|\br{\nu^N_t}{\varphi}\right|\leq\int_W\!|\varphi(w)|\,\nu^N_t(dw)
  \leq C\norm\varphi_\chf\int_W\!\rho_4(w)\,\nu^N_t(dw) \leq
  C\norm\varphi_\chf\sqrt{\br{\nu^N_t}{\rho^2_4}},\] and the term inside the square root
  is almost surely bounded by \eqref{eq:mom2}, so the mapping
  $\varphi\in\chf\longmapsto\br{\nu^N_t}{\varphi}$ is continuous almost surely.
 
  Next we need to show that $\br{\nu^N_t}{\varphi}$ satisfies \eqref{eq:intForm} for all
  $\varphi\in\chf$. That is, we need to show that the formula
  \begin{equation}\label{eq:defMgGen}
    M^{N,\varphi}_t=\br{\nu^N_t}{\varphi}-\br{\nu^N_0}{\varphi}
    -\int_0^t\!\int_W\!\int_W\!\int_W\!\Lambda\varphi(w_1,w_2;z)
    \,\nu^N_s(dz)\,\nu^N_s(dw_2)\,\nu^N_s(dw_1)\,ds   
  \end{equation}
  defines a martingale for each $\varphi\in\chf$. Let $\varphi\in\chf$ and $m>0$ and write
  $(\varphi\wedge m)(w)=\varphi(w)\wedge m$. $\varphi\wedge m$ is in $\mw$, so
  $M^{N,\varphi\wedge m}_t$ is a martingale. We deduce that given any $0\leq
  s_1\leq\dotsm\leq s_k<s<t$ and any continuous bounded functions $\psi_1,\dotsc,\psi_k$
  on $\chf$, if we
  let \[X^m=\psi_1(\nu^N_{s_1})\dotsm\psi_k(\nu^N_{s_k})\big[M^{N,\varphi\wedge m}_t
  -M^{N,\varphi\wedge m}_s\big],\] then $\ee(X^m)=0$. Using the Monotone Convergence
  Theorem one can show that
  $X^m\to\psi_1(\nu^N_{s_1})\dotsm\psi_k(\nu^N_{s_k})\big[M^{N,\varphi}_t
  -M^{N,\varphi}_s\big]$ as $m\to\infty$. On the other hand, the sequence $(X^m)_{m>0}$ is
  uniformly integrable. Indeed, using \eqref{assum:rho:h} and \eqref{eq:mom2} one can show
  that
  \[\ee\!\left(\left|\big(\psi_1(\nu^N_{s_1})\dotsm\psi_k(\nu^N_{s_k})\big[M^{N,\varphi\wedge m}_t
      -M^{N,\varphi\wedge m}_s\big]\right|^2\right) \leq
  Ct^2\ee\!\left(\sup_{r\in[0,t]}\br{\nu^N_r}{\rho^2_4}\right)<\infty.\] We deduce
  that \[\ee(\psi_1(\nu^N_{s_1})\dotsm\psi_k(\nu^N_{s_k})\big[M^{N,\varphi}_t
  -M^{N,\varphi}_s\big])=\lim_{m\to\infty}\ee(X^m)=0,\] which implies that
  $M^{N,\varphi}_t$ is a martingale. The fact that
  $\br{M^{N,\varphi_1}}{M^{N,\varphi_2}}_t$ has the right form follows from the same
  arguments as those for \eqref{eq:formCoQuadN} (here we need to replace $\varphi_1$ and
  $\varphi_2$ by $\varphi_1^m$ and $\varphi_2^m$ and then take $m\rightarrow\infty$ as
  above).
\end{proof}

\subsubsection{The drift term}

By Proposition \ref{prop:ext}, we have now that the fluctuations process $\sigma^N_t$ is
well defined as a process taking values in $\chfp$ and it satisfies
\begin{equation}\label{eq:orFormSigmaN}
  \begin{aligned}
    &\br{\sigma^N_t}{\varphi}=\sqrt{N}\br{\nu^N_0-\nu_0}{\varphi}
    +\sqrt{N}M^{N\varphi}_t\\
    &\quad+\sqrt{N}\int_0^t\!\int_W\!\int_W\!\int_W\!\Lambda\varphi(w_1,w_2;z)
    \Big[\nu^N_s(dz)\nu^N_s(dw_2)\nu^N_s(dw_1)-\nu_s(dz)\nu_s(dw_2)\nu_s(dw_1)\Big]ds
  \end{aligned}
\end{equation}
for every $\varphi\in\chf$. The integral term can be rewritten as
\begin{multline}
\int_0^t\!\int_W\!\int_W\!\int_W\!\Lambda\varphi(w_1,w_2;z)
  \Big[\sigma^N_s(dz)\nu^N_s(dw_2)\nu^N_s(dw_1)\\
  +\nu_s(dz)\big(\sigma^N_s(dw_2)\nu^N_s(dw_1)+\nu_s(dw_2)\sigma^N_s(dw_1)\big)\Big]ds.
 \end{multline}
Therefore,
\begin{equation}\label{eq:formSigmaN}
  \br{\sigma^N_t}{\varphi}=\sqrt{N}\br{\nu^N_0-\nu_0}{\varphi}+\sqrt{N}M^{N,\varphi}_t
  +\int_0^t\!\br{\sigma^N_s}{J^N_s\varphi}\,ds
\end{equation}
for each $\varphi\in\chf$, where
\begin{multline}\label{eq:defJN}
  J^N_s\varphi(z)=\int_W\!\int_W\!\Lambda\varphi(w_1,w_2;z)\,\nu^N_s(dw_2)\,\nu^N_s(dw_1)\\
  +\int_W\!\int_W\!\Lambda\varphi(w,z;x)\,\nu^N_s(dw)\,\nu_s(dx)
  +\int_W\!\int_W\!\Lambda\varphi(z,w;x)\,\nu_s(dw)\,\nu_s(dx).
\end{multline}

Observe that $J^N_s=J_{\nu^N_s,\nu_s}$ and $J_s=J_{\nu_s,\nu_s}$, where the operators
$J_{\mu_1,\mu_2}$ are the ones defined in Assumption \ref{assum:2clt}. Hence
\eqref{assum:2clt:3} and Proposition \ref{prop:bdMoments} imply that $J^N_s$ and $J_s$ are
bounded linear operators on each space $\cci$ ($i=0,2,3$) and, moreover, for all
$\varphi\in\cci$,
\begin{equation}
  \norm{J^N_s\varphi}_\cci\leq C\norm\varphi_\cci
  \qquad\text{and}\qquad\norm{J_s\varphi}_\cci\leq C\norm\varphi_\cci,\label{eq:bdDrifts}
\end{equation}
almost surely for some constant $C>0$ independent of $N$ and $s$. Similarly, given any
$\varphi\in\ccz$,
\begin{equation}
  \norm{\left(J^N_s-J_s\right)\varphi}_\ccz\leq
  C\norm\varphi_\ccz\norm{\nu^N_s-\nu_s}_\ccdp\label{eq:lipDrifts}
\end{equation}
almost surely for some constant $C>0$ independent of $N$ and $s$.

\subsubsection{Uniform estimate for the martingale term in
  \texorpdfstring{$\chfp$}{H\textunderscore4'}}\label{sec:unifBdMg}

Proposition \ref{prop:ext} implies that the martingale term $M^{N,\varphi}_t$ is
well defined for all $\varphi\in\chf$. We will denote by $M^N_t$ the bounded linear
functional on $\chf$ given by $M^N_t(\varphi)=M^{N,\varphi}_t$.

\begin{teosec}\label{thm:bdUnifMgN}
  $\sqrt{N}M^N_t$ is a c\`adl\`ag square integrable martingale in $\chfp$, whose
  Doob--Meyer process
  $\big(\big\langle\!\!\big\langle\sqrt{N}M^N\big\rangle\!\!\big\rangle_t(\varphi_1)\big)(\varphi_2)
  =N\big<\sqrt{N}M^N(\varphi_1),\sqrt{N}M^N(\varphi_2)\big>_t$ (which is a
  linear operator from $\chf$ to $\chfp$) can be obtained from the formula in \eqref{eq:formCoQuadN}.
  Moreover,
  \[\sup_{N>0}\ee\!\left(\sup_{t\in[0,T]}\left\|\sqrt{N}M^N_t\right\|^2_\chfp\right)<\infty.\]
\end{teosec}

\begin{proof}
  We already know, by Proposition \ref{prop:ext}, that $\sqrt{N}M^N_t$ is a martingale in
  $\chfp$ with the right Doob--Meyer process. The fact that the paths of $\sqrt{N}M^N_t$
  are in $D([0,T],\chfp)$ can be checked by the same arguments as those in the proof of
  Corollary 3.8 in \citet{meleard}. So we only need to show the last assertion. Let
  $(\phi_k)_{k\geq0}$ be an orthonormal complete basis of $\chf$. We observe that, by 
  \eqref{assum:rho:h}, if $\chi_w\in\chfp$ is defined by $\chi_w(\varphi)=\varphi(w)$ then
  \[\sum_{k\geq0}\phi_k^2(w)=\|\chi_w\|^2_\chfp\leq C\rho^2_4(w).\]
  Thus by Proposition \ref{prop:ext} and Doob's inequality,
  \begin{align*}
    \ee\bigg(\sup_{t\in[0,T]}&\left\|\sqrt{N}M^N_t\right\|^2_\chfp\bigg)
    \leq\ee\!\left(\sum_{k\geq0}\sup_{t\in[0,T]}N\left|M^{N,\phi_k}_t\right|^2\right)
    \leq4\sum_{k\geq0}\ee\!\left(N\big<M^{N,\phi_k},M^{N,\phi_k}\big>_T\right)\\
    &=4\ee\Bigg(\int_0^T\!\int_W\!\int_W\!\int_W\!\int_{W\!\times\!W}
    \!\sum_{k\geq0}\big(\phi_k(w_1')-\phi_k(w_1)
    +\phi_k(w_2')-\phi_k(w_2)\big)^2\\
    &\hspace{1.6in}\cdot\Lambda(w_1,w_2,z,dw_1'\ootimes dw_2')
    \,\nu^N_s(dz)\,\nu^N_s(dw_2)\,\nu^N_s(dw_1)\,ds\Bigg)\\
    &\leq C\int_0^T\!\ee\Bigg(\int_W\!\int_W\!\int_W\!\int_{W\!\times\!W}\!
    \big(\rho^2_4(w_1)+\rho^2_4(w_2)+\rho^2_4(w_1')+\rho^2_4(w_2')\big)\\
    &\hspace{1.6in}\cdot\Lambda(w_1,w_2,z,dw_1'\ootimes dw_2')
    \,\nu^N_s(dz)\,\nu^N_s(dw_2)\,\nu^N_s(dw_1)\Bigg)\,ds\\
    &\leq C\int_0^T\!\ee\Bigg(\int_W\!\int_W\!\int_W\!
    \big(2\rho^2_4(w_1)+2\rho^2_4(w_2)+\rho^2_4(z)\big)
    \,\nu^N_s(dz)\,\nu^N_s(dw_2)\,\nu^N_s(dw_1)\Bigg)\,ds\\
    &\leq C\int_0^T\!\ee\!\left(\br{\nu^N_s}{\rho^2_4}\right)ds.
  \end{align*}
  The last integral is bounded, uniformly in $N$, by Proposition \ref{prop:bdMoments}.
\end{proof}

\subsubsection{Evolution equation for \texorpdfstring{$\sigma^N_t$}{\textbackslash
    sigma\textcircumflex N\textunderscore t} in \texorpdfstring{$\chtp$}{H\textunderscore3'}}

Recall that our goal is to prove convergence of $\sigma^N_t$ in $D([0,T],\chup)$. Therefore, a
necessary previous step is to make sense of \eqref{eq:formSigmaN} as an equation in
$\chup$. We will actually need to show something stronger: $\sigma^N_t$ can be seen as a
semimartingale in $\chtp$, whose semimartingale decomposition takes the form suggested by
\eqref{eq:formSigmaN}. We need the following simple result first (for its proof see
Proposition 3.4 of \citet{meleard}):

\begin{lem}\label{lem:bdSigmaNLoc}
  For every $N>0$ there is a constant $C(N)>0$ such that
  \[\sup_{t\in[0,T]}\ee\!\left(\|\sigma^N_t\|_\chfp\right)\leq C(N).\]
\end{lem}

Recall that under our assumptions, $J^N_s$ need not be (and in general is not) a bounded
operator on $\cht$, nor on any other $\chii$, and in fact $J^N_s(\chii)$ need not even be
contained in $\chii$, so it does not make complete sense to speak of
$\left(J^N_s\right)^*$ as the adjoint operator of $J^N_s$. Nevertheless, for convenience
we will abuse notation by writing $\left(J^N_s\right)^*\sigma^N_s$ to denote the linear
functional defined by the following mapping:
\[\varphi\in\cht\longmapsto\left(J^N_s\right)^*\sigma^N_s(\varphi)=\br{\sigma^N_s}{J^N_s\varphi}\in\rr.\]
Part of the proof of the following result will consist in showing that
$\left(J^N_s\right)^*\sigma^N_s$ is actually in $\chtp$.

\begin{prop}\label{prop:bochN}
  For each $N>0$, $\sigma^N_t$ is an $\chtp$-valued semimartingale, and its Doob--Meyer
  decomposition is given by
  \begin{equation}\label{eq:bochN}
    \sigma^N_t=\sigma^N_0+\sqrt{N}M^N_t+\int_0^t\!\left(J^N_s\right)^*\sigma^N_s\,ds,
  \end{equation}
  where the above is a Bochner integral in $\chtp$.
\end{prop}

\begin{proof}
  By Theorem \ref{thm:bdUnifMgN} and the embedding $\chfp\hookrightarrow\chtp$,
  $\sqrt{N}M^N_t$ is an $\chtp$-valued martingale. Thus, by \eqref{eq:formSigmaN}, the
  only thing we need to show is that the integral term makes sense as a Bochner integral
  in $\chtp$. The first step in doing this is to show that
  $\left(J^N_s\right)^*\sigma^N_s\in\chtp$ for all $s\in[0,T]$. That is, we need to show
  that there is a $C>0$ such that
  \begin{equation}
    \left|\br{\sigma^N_s}{J^N_s\varphi}\right|\leq C\norm{\varphi}_\cht\label{eq:bdJNSigmaN}
  \end{equation}
  for all $\varphi\in\cht$. Observe that by \eqref{eq:bdDrifts} and the embedding
  $\cht\hookrightarrow\cct$, $J^N_s\varphi\in\cct$ for $\varphi\in\cht$, and thus
  \[\left|\br{\sigma^N_s}{J^N_s\varphi}\right|
  \leq\norm{\sigma^N_s}_\cctp\norm{J^N_s\varphi}_\cct \leq
  C\norm{\sigma^N_s}_\cctp\norm{\varphi}_\cct \leq
  C\norm{\sigma^N_s}_\cctp\norm{\varphi}_\cht\] for such a function $\varphi$ by
  \eqref{eq:normIneq}, so \eqref{eq:bdJNSigmaN} holds almost surely by Lemma
  \ref{lem:bdSigmaNLoc} and \eqref{eq:normIneqD}.

  To see that the Bochner integral is (almost surely) well defined, we recall (see Section
  V.5 in \citet{yosida}) that it is enough to prove that: (i) given any function $F$ in
  the dual of $\chtp$, the mapping $s\longmapsto F\big(\big(J^N_s\big)^*\sigma^N_s\big)$
  is measurable; and (ii)
  $\int_0^T\!\big\|\big(J^N_s\big)^*\sigma^N_s\big\|_\chtp\,ds<\infty$. (i) is satisfied
  by the continuity assumptions on the parameters and (ii) follows from
  \eqref{eq:bdJNSigmaN}, using the fact that the constant $C$ there can be chosen
  uniformly in $s$.
\end{proof}

We omit the proof of the following corollary (see Corollary 3.8 of \citet{meleard}):

\begin{cor}
  For any $N>0$, the process $\sigma^N_t$ has paths in $D([0,T],\chtp)$.
\end{cor}

\subsubsection{Uniform estimate for \texorpdfstring{$\sigma^N_t$}{\textbackslash
    sigma\textcircumflex N\textunderscore t} on
  \texorpdfstring{$\ccdp$}{C\textunderscore2'}}

Having given sense to equation \eqref{eq:bochN} in $\chtp$, we can now give a uniform
estimate for $\sigma^N_t$ in $\ccdp$. This will be crucial for obtaining the tightness of
$\sigma^N_t$ in the proof of Theorem \ref{thm:clt}.

\begin{teosec}\label{thm:bdUnifSigmaN}
  \[\sup_{N>0}\sup_{t\in[0,T]}\ee\!\left(\left\|\sigma^N_t\right\|^2_\ccdp\right)<\infty.\]
\end{teosec}

\begin{proof}
  By \eqref{eq:bochN} and the embedding $\chtp\hookrightarrow\ccdp$,
  \begin{equation}
    \ee\!\left(\left\|\sigma^N_t\right\|^2_\ccdp\right)
    \leq2\ee\!\left(\left\|\sigma^N_0\right\|^2_\ccdp\right)
    +2\ee\!\left(\left\|\sqrt{N}M^N_t\right\|^2_\ccdp\right)
    +2\ee\!\left(\left\|\int_0^t\!\left(J^N_s\right)^*\sigma^N_s
        \,ds\right\|^2_\ccdp\right).
  \end{equation}
  The first expectation on the right side is bounded uniformly in $N$ by
  \eqref{eq:iniFluc}, and the same holds for the second one by \eqref{eq:normIneqD} and
  Theorem \ref{thm:bdUnifMgN}. For the last expectation we have
  \begin{multline}
    \ee\!\left(\left\|\int_0^t\!\left(J^N_s\right)^*\sigma^N_s \,ds\right\|^2_\ccdp\right)
    \leq\ee\!\left(\left[\int_0^t\!
        \left\|\left(J^N_s\right)^*\sigma^N_s\right\|_\ccdp ds\right]^2\right)\\
    \leq
    T\int_0^t\!\ee\!\left(\left\|\left(J^N_s\right)^*\!\sigma^N_s\right\|^2_\ccdp\right)\,ds
    \leq CT\int_0^T\!\ee\!\left(\sup_{s\in[0,t]}\left\|\sigma^N_s\right\|^2_\ccdp\right)dt,
  \end{multline}
  where we used Corollary V.5.1 of \citet{yosida} in the first inequality and
  \eqref{eq:bdDrifts} in the last one.  Thus by Gronwall's Lemma we get
  $\ee\!\left(\sup_{t\in[0,T]}\left\|\sigma^N_t\right\|^2_\ccdp\right)\leq C_1e^{C_2T}$,
  uniformly in $N$, and the result follows.
\end{proof}

\subsubsection{Proof of the theorem}

We are finally ready to prove Theorem \ref{thm:clt}.

\begin{proof}[Proof of Theorem \ref{thm:clt}]
  As before, we will proceed in several steps.

  \pparagraph{Step 1.} Our first goal is to show that the sequence of processes
  $\sigma^N_t$ is tight in $D([0,T],\chup)$. By Aldous' criterion (which we take from
  Theorem 2.2.2 in \citet{joffeMetiv} and the corollary that precedes it in page 34), we
  need to prove that the following two conditions hold:
  \begin{subequations}
    \begin{enumerate}[label=(t\arabic*),topsep=2pt,itemsep=-8pt]
    \item\label{wkC1} For every rational $t\in[0,T]$ and every $\ep>0$, there is a compact
      $K\subseteq \chup$ such that
      \[\sup_{N>0}\pp\!\left(\sigma^N_t\notin K\right)\leq\ep.\]
    \item\label{wkC2} If $\mathfrak T^N_T$ is the collection of stopping times with
      respect to the natural filtration associated to $\sigma^N_t$ that are almost surely
      bounded by $T$, then for every $\ep>0$,
      \[\lim_{r\rightarrow0}\limsup_{N\rightarrow\infty}
      \sup_{\substack{s<r\\\tau\in\mathfrak
          T^N_T}}\pp\big(\norm{\sigma^N_{\tau+s}-\sigma^N_\tau}_\chup>\ep\big)=0.\]
    \end{enumerate}
  \end{subequations}

  Observe that since the embedding of $\chdp$ into $\chup$ is compact, \ref{wkC1} will
  follow once we show that for any $\ep>0$ and $t\in[0,T]$ there is an $L>0$ such that
  $\sup_{N>0}\pp\!\left(\left\|\sigma^N_t\right\|_\chdp>L\right)<\ep$. This follows
  directly from Markov's inequality, \eqref{eq:normIneqD}, and Theorem
  \ref{thm:bdUnifSigmaN}, since given any $\ep>0$,
  \[\sup_{N>0}\pp\!\left(\left\|\sigma^N_t\right\|_\chdp>L\right)
  \leq\frac{1}{L^2}\sup_{N>0}\ee\!\left(\left\|\sigma^N_t\right\|^2_\chdp\right)
  \leq\frac{1}{L^2}\sup_{N>0}\ee\!\left(\left\|\sigma^N_t\right\|^2_\ccdp\right) <\ep\]
  for large enough $L$.

  To obtain \ref{wkC2} we will use the semimartingale decomposition of $\sigma^N_t$
  in $\chtp$ given in Proposition \ref{prop:bochN}, i.e.,
  $\sigma^N_t=\sigma^N_0+\sqrt{N}M^N_t+\int_0^t\!\left(J^N_s\right)^*\!\sigma^N_s\,ds$.
  By Rebolledo's criterion (see Corollary 2.3.3 in \citet{joffeMetiv}), \ref{wkC2}
  is obtained for the martingale term $\sqrt{N}M^N_t$ if it is proved for the trace of its
  Doob--Meyer process $\cramped{\big\langle\!\!\big\langle
      \sqrt{N}M^N\big\rangle\!\!\big\rangle_t}$ in $\chu$, and thus for $\sigma^N_t$ if
  it is proved moreover for the finite variation term
  $\int_0^t\!\left(J^N_s\right)^*\!\sigma^N_s\,ds$ ($\sigma^N_0$ is tight by hypothesis).

  We start with the martingale part. Let $\tau$ be a stopping time bounded by $T$ and let
  $s>0$.  Let $(\phi_k)_{k\geq0}$ be an orthonormal complete basis of $\chu$. Using
  the same calculations as in the proof of Theorem \ref{thm:bdUnifMgN} we get
  \begin{align*}
    \ee\Big(\Big|\text{tr}_\chu\big\langle&\!\!\big\langle
    \sqrt{N}M^N\big\rangle\!\!\big\rangle_{\tau+s}
    -\text{tr}_\chu\big\langle\!\!\big\langle
    \sqrt{N}M^N\big\rangle\!\!\big\rangle_\tau\Big|\Big)\\
    &=\ee\bigg(\int_\tau^{\tau+s}\!\int_W\!\int_W\!\int_W\!\int_{W\!\times\!W}\!\sum_{k\geq0}
    \big(\phi_k(w_1')-\phi_k(w_1)+\phi_k(w_2')-\phi_k(w_2)\big)^2\\
    &\hspace{1.4in}\cdot\Lambda(w_1,w_2,z,dw_1'\otimes
    dw_2')\,\nu^N_s(dz)\,\nu^N_s(dw_2)\,\nu^N_s(dw_1)\bigg)\\
    &\leq Cs,
  \end{align*}
  uniformly in $N$. Thus by Markov's inequality,
  \[\pp\Big(\Big|\text{tr}_\chu\big\langle\!\!\big\langle
  \sqrt{N}M^N\big\rangle\!\!\big\rangle_t -\text{tr}_\chu\big\langle\!\!\big\langle
  \sqrt{N}M^N\big\rangle\!\!\big\rangle_t\Big|>\ep\Big) \leq\frac{1}{\ep}Cs,\] whence
  \ref{wkC2} follows for the martingale term.

  For the integral term we have that
  \begin{multline}
    \ee\bigg(\bigg\| \int_0^{\tau+s}\!\left(J^N_r\right)^*\!\sigma^N_r\,dr
    -\int_0^\tau\!\left(J^N_r\right)^*\!\sigma^N_r\,dr\bigg\|_\chup\bigg)
    \leq\ee\!\left(\int_\tau^{\tau+s}\!
      \left\|\left(J^N_r\right)^*\!\sigma^N_r\right\|_\ccdp\,dr\right)\\
    \leq C\int_\tau^{\tau+s}\!\ee\!\left(\left\|\sigma^N_r\right\|_\ccdp\right)dr \leq
    Cs\sup_{r\in[0,T]}\sqrt{\ee\!\left(\left\|\sigma^N_r\right\|^2_\ccdp\right)}
  \end{multline}
  for some $C>0$, uniformly in $N$, where we used Corollary V.5.1 of \citet{yosida} as
  before and \eqref{eq:normIneqD} in the first inequality and \eqref{eq:bdDrifts} in the
  second one. Using Markov's inequality as before and Theorem \ref{thm:bdUnifSigmaN} we
  obtain \ref{wkC2} for the integral term.
   
  \pparagraph{Step 2.} We have now that every subsequence of $\sigma^N_t$ has a further
  subsequence which converges in distribution in $D([0,T],\chup)$.  Consider a convergent
  subsequence of $\sigma^N_t$, which we will still denote by $\sigma^N_t$, and let
  $\sigma_t$ be its limit in $D([0,T],\chup)$. Observe that the only jumps of $\sigma^N_t$
  are those coming from $\nu^N_t$ and, with probability 1, at most two agents jump at the
  same time.  Suppose that there is a jump at time $t$, involving agents $i$ and $j$. Then
  given $\varphi\in\chu$,
  \begin{align}
    \left|\br{\sigma^N_t}{\varphi}-\br{\sigma^N_{t-}}{\varphi}\right|
    &=\frac{1}{\sqrt{N}}\left|\varphi(\eta^N_t(i))+\varphi(\eta^N_t(j))
      -\varphi(\eta^N_{t-}(i))-\varphi(\eta^N_{t-}(j))\right|\\
    &\leq\frac{C}{\sqrt{N}}\norm\varphi_\chu
    \left[\sup_{r\in[0,t]}\rho_1(\eta^N_r(i))+\sup_{r\in[0,t]}\rho_1(\eta^N_r(j))\right]
  \end{align}
  by \eqref{assum:rho:h}. We deduce by \eqref{eq:mom2EtaN} that
  \begin{equation}\label{eq:strContClt}
    \ee\!\left(\sup_{s\in[0,t]}
      \left\|\sigma^N_s-\sigma^N_{s-}\right\|^2_\chup\right)\leq\frac{C}{N}
  \end{equation}
  and hence $\sup_{s\in[0,t]}\left\|\sigma^N_s-\sigma^N_{s-}\right\|_\chup$ converges in
  probability to 0 as $N\rightarrow\infty$. Therefore, $\sigma_t$ is almost surely
  strongly continuous by Proposition 3.26 of \cite{jacodShir}. That is, we have shown that
  every limit point of $\sigma^N_t$ is (almost surely) in $C([0,T],\chup)$.

  \pparagraph{Step 3.} Our next goal is to prove that the sequence of martingales
  $\sqrt{N}M^N_t$ converges in distribution in $D([0,T],\chup)$ to the centered Gaussian
  process $Z_t$ defined in the statement of the theorem.  That is, we need to show that
  given any $\varphi_1,\varphi_2\in\chu$, the sequence of $\rr^2$-valued martingales
  $\sqrt{N}M^{N,(\varphi_1,\varphi_2)}_t=\left(\sqrt{N}M^{N,\varphi_1}_t,\sqrt{N}M^{N,\varphi_2}_t\right)$
  converges in distribution to $(Z_t(\varphi_1),Z_t(\varphi_2))$.

  By \eqref{eq:bochN}, $\sqrt{N}M^N_t$ and $\sigma^N_t$ have the same jumps, and thus
  \eqref{eq:strContClt} implies that
  \begin{equation}
  \ee\!\left(\sup_{s\in[0,t]}\left|\sqrt{N}M^{N,(\varphi_1,\varphi_2)}_s
      -\sqrt{N}M^{N,(\varphi_1,\varphi_2)}_{s-}\right|^2\right)
  \xrightarrow[N\rightarrow\infty]{}0.\label{eq:mgJumps}
  \end{equation}
  On the other hand, we claim that for every $\varphi_1,\varphi_2\in\chu$,
  \begin{equation}
    \lim_{N\to\infty}\ee\!\left(\Big<\sqrt{N}M^{N,\varphi_1},\sqrt{N}M^{N,\varphi_2}\Big>_t\right)
    =\int_0^t\!C^{\varphi_1,\varphi_2}_s\,ds.\label{eq:cvQC}
  \end{equation}
  \eqref{eq:mgJumps} and \eqref{eq:cvQC} imply that
  $\sqrt{N}M^{N,(\varphi_1,\varphi_2)}_t$ satisfies the hypotheses of the Martingale
  Central Limit Theorem (see Theorem VII.1.4 in \citet{ethKur}) so, assuming that
  \eqref{eq:cvQC} holds, we get that $\sqrt{N}M^{N,(\varphi_1,\varphi_2)}_t$ converges in distribution in $D([0,T],\rr^2)$ to
  $(Z_t(\varphi_1),Z_t(\varphi_2))$.

  To prove \eqref{eq:cvQC} it is enough to consider the case
  $\varphi_1=\varphi_2=\varphi$, the general case follows by polarization. Given
  $\mu\in D([0,T],\chup)$ let
  \begin{multline}
    \Psi_t(\mu)=\int_0^t\!\int_W\!\int_W\!\int_W\!\int_{W\!\times\!W}\!
    \difd^2\,\Lambda(w_1,w_2,z,dw_1'\ootimes dw_2')\\
    \cdot\mu_s(dz)\,\mu_s(dw_2)\,\mu_s(dw_1)\,ds.
  \end{multline}
  Then we need to prove that $\lim_{N\to\infty}\ee(\Psi_t(\nu^N))=\Psi_t(\nu)$. Let $p>1$
  be the exponent we assumed to be such that $\rho^p_1\leq C\rho_4$ for some $C>0$.. Repeating the
  calculations in the proof of Theorem \ref{thm:bdUnifMgN} and using Jensen's inequality we get that
  \[|\Psi_t(\nu^N)|^p\leq\Big[C_1t\norm\varphi^2_\chu\sup_{s\in[0,t]}\br{\nu^N_s}{\rho^2_1}\Big]^p
  \leq C_2t^p\norm\varphi^{2p}_\chu\sup_{s\in[0,t]}\br{\nu^N_s}{\rho^2_4}.\]
  Thus Proposition \ref{prop:bdMoments} implies that the sequence
  $\big(\Psi_t(\nu^N)\big)_{N>0}$ is uniformly integrable, whence we deduce the desired
  convergence.
  
  \pparagraph{Step 4.} As in Step 2, let $\sigma_t$ be a limit point of
  $\sigma^N_t$. Observe that by the embedding $\chup\hookrightarrow\cczp$, $\sigma^N_t$
  converges in distribution to $\sigma_t$ in $D([0,T],\cczp$). We want to prove now that
  $\sigma_t$ satisfies \eqref{eq:cltWeak}.

  Fix $\varphi\in\ccz$. By \eqref{eq:bochN},
  \begin{equation}
    \begin{aligned}
      \br{\sigma_t}{\varphi}-&
      \br{\sigma_0}{\varphi}-\int_0^t\!\br{\sigma_s}{J_s\varphi}ds
      -Z_t(\varphi)\\
      &=\left[\sqrt{N}M^{N,\varphi}_t-Z_t(\varphi)\right]
      +\left[\br{\sigma_t}{\varphi}-\br{\sigma^N_t}{\varphi}\right]
      +\left[\br{\sigma^N_0}{\varphi}-\br{\sigma_0}{\varphi}\right]\\
      &\qquad+\int_0^t\!\left[\br{\sigma^N_s}{J^N_s\varphi}-\br{\sigma^N_s}{J_s\varphi}\right]ds
      +\int_0^t\!\left[\br{\sigma^N_s}{J_s\varphi}-\br{\sigma_s}{J_s\varphi}\right]ds,
    \end{aligned}\label{eq:sigmaSigmaN}
  \end{equation}
  so we need to show that the right side converges in distribution to 0 as
  $N\rightarrow\infty$. The first term goes to 0 by the previous step. The next two go to
  0 because $\sigma_t$ is a limit point of $\sigma^N_t$ and, since $J_s\varphi\in\ccz$,
  the last term goes to 0 for the same reason.

  To show that the remaining term in \eqref{eq:sigmaSigmaN} also goes to 0 in
  distribution, it is enough to show that
  \begin{equation}
    \ee\!\left(\left|\int_0^t\!\br{\sigma^N_s}{\left(J^N_s-J_s\right)\varphi}ds
      \right|\right)\xrightarrow[N\rightarrow\infty]{}0.\label{eq:JNJTo0}
  \end{equation}
  Since, by \eqref{eq:bdDrifts}, $J^N_s-J_s$ maps $\ccz$ into itself, we get by
  using \eqref{eq:normIneqD} and \eqref{eq:lipDrifts} that
  \begin{align}
    \left|\br{\sigma^N_s}{\left(J^N_s-J_s\right)\varphi}\right|
  &\leq\norm{\sigma^N_s}_\cczp\norm{\left(J^N_s-J_s\right)\varphi}_\ccz
  \leq C\norm{\sigma^N_s}_\ccdp\norm{\varphi}_\ccz\norm{\nu^N_s-\nu_s}_\ccdp\\
  &=\frac{C}{\sqrt{N}}\norm{\varphi}_\ccz\norm{\sigma^N_s}_\ccdp^2.
  \end{align}
  \eqref{eq:JNJTo0} now follows from this bound and Theorem \ref{thm:bdUnifSigmaN}.

  \pparagraph{Step 5.}  We have shown in Step 4 that if $\sigma_t$ is any accumulation
  point of $\sigma^N_t$, then $\sigma_t$ satisfies \eqref{eq:cltWeak} for every
  $\varphi\in\ccz$. To see that the limit points of $\sigma^N_t$ actually solve
  \eqref{eq:clt}, the only thing left to show is that the integral term in \eqref{eq:clt}
  makes sense as a Bochner integral in $\cczp$. This can be verified by repeating the
  arguments of the proof of Proposition \ref{prop:bochN}.

  \pparagraph{Step 6.} We want to prove now pathwise uniqueness for the solutions of
  \eqref{eq:clt}. Fix a centered Gaussian process $Z_t$ in $\cczp$ with the right
  covariance structure and suppose that $\sigma^1_t, \sigma^2_t\in\ccz$ are two solutions of
  \eqref{eq:clt} for this choice of $Z_t$. Then
  $\sigma^1_t-\sigma^2_t=\int_0^t\!\left(J_s^*\sigma^1_s-J_s^*\sigma^2_s\right)ds$,
  so
  \[\sup_{t\in[0,T]}\left\|\sigma^1_t-\sigma^2_t\right\|_\cczp
  \leq\int_0^T\!\sup_{s\in[0,t]}\norm{J_s^*\left(\sigma^1_s-\sigma^2_s\right)}_\cczp\,dt.\]
  By \eqref{eq:bdDrifts}, $J_s$ is a bounded operator on $\ccz$, and thus
  so is $J_s^*$ as an operator on $\cczp$. Moreover, $\norm{J_s^*}_\cczp$ can
  be bounded uniformly in $s$.  Thus
  \[\ee\!\left(\sup_{t\in[0,T]}\left\|\sigma^1_t-\sigma^2_t\right\|_\cczp\right)
  \leq C\int_0^t\!\ee\!\left(\sup_{s\in[0,t]}\|\sigma^1_s-\sigma^2_s\|_\cczp\right)dt,\]
  and Gronwall's Lemma implies that $\sigma^1_t=\sigma^2_t$ for all $t\in[0,T]$ almost
  surely, so the pathwise uniqueness for \eqref{eq:clt} follows.

  \pparagraph{Step 7.}  We have now that any accumulation point $\sigma_t$ of the sequence
  $\sigma^N_t$ satisfies equation \eqref{eq:clt}, which has a unique pathwise
  solution. The last thing left to show is the uniqueness in law for the solutions of this
  equation. Since we have pathwise uniqueness, this can be obtained by adapting the
  Yamada--Watanabe Theorem to our setting (see Theorem IX.1.7 of \citet{revYor}). The proof
  works in the same way assuming we can construct regular conditional probabilities in
  $D([0,T],\cczp)$, which is possible in any complete metric space (see Theorem I.4.12 of
  \citet{durrett}). This (together with the embedding $\chup\hookrightarrow\cczp$) implies
  that \eqref{eq:clt} determines a unique process in $C([0,T],\chup)$.
\end{proof}

\subsection{Proof of Theorems \ref{thm:iniProd} and
  \ref{thm:type:finite}-\ref{thm:type:rrd}}

\begin{proof}[Proof of Theorem \ref{thm:iniProd}]
  There are three conditions to check. The first one, $\sigma^N_0\Longrightarrow\sigma_0$
  in $\chup$, follows directly from applying the Central Limit Theorem in $\rr$ to each of
  the processes $\br{\sigma^N_0}{\varphi}$ for $\varphi\in\chu$, while the condition
  $\sup_{N>0}\ee\!\left(\br{\nu^N_0}{\rho^2_4}\right)<\infty$ is straightforward.
  For the remaining one we can prove something stronger, namely that 
  $\sup_{N>0}\ee\!\left(\norm{\sigma^N_0}^2_\chfp\right)<\infty$. 
  In fact, if $(\phi_k)_{k\geq0}$ is a complete orthonormal basis of $\chf$ and $\eta^N_0$
  is chosen by picking the type $\eta^N_0(i)$ of each agent $i\in I_N$ independently
  according to $\nu_0$ then
  \[\ee\!\left(\norm{\sigma^N_0}^2_\chfp\right)
  =\ee\bigg(\sum_{k\geq0}\br{\sigma^N_0}{\phi_k}^2\bigg)
  =\frac{1}{N}\sum_{k\geq0}\ee\bigg(\bigg[\sum_{i=1}^N\left[\phi_k(\eta^N_0(i))
        -\br{\nu_0}{\phi_k}\right]\bigg]^2\bigg).\] A simple computation and
  \eqref{assum:rho:h} (see the proof of Proposition 3.5 in \citet{meleard}) show that this is
  bounded by $\ee\!\left(\br{\nu^N_0}{\rho^2_4}\right)+\br{\nu_0}{\rho^2_4}$, which is in turn
  bounded by some $C<\infty$ uniformly in $N$, so the result follows.
\end{proof}

For Theorems \ref{thm:type:finite} (finite $W$), \ref{thm:type:omega}
($W=\Omega\subseteq\rr^d$ smooth and compact), and \ref{thm:type:rrd} ($W=\rr^d$), we
already explained why the assumptions of Theorem \ref{thm:clt} hold, so the results follow
directly from that theorem (together with \eqref{eq:finDim} when $W$ is finite). We are
left with the case $W=\zz^d$.

\begin{proof}[Proof of Theorem \ref{thm:type:zzd}]
  Let $\varphi\in\ell^\infty(\zz^d)$. Then
  \[\norm\varphi^2_{2,D}=\sum_{x\in\zz^d}\frac{\varphi(x)^2}{1+|x|^{2D}}
  \leq C\norm\varphi^2_\infty,\] where we used the fact that $2D>d$ implies that
  $\sum_{x\in\zz^d}(1+|x|^{2D})^{-1}<\infty$. This gives the embedding
  $\ell^\infty(\zz^d)\hookrightarrow\ell^{2,D}(\zz^d)$. The other continuous embeddings in
  \eqref{eq:embZzd} are similar. To see that the embedding
  $\ell^{2,D}(\zz^d)\hookrightarrow\ell^{2,2D}$ is compact, observe that the family
  $(e_y)_{y\in\zz^d}\subseteq\ell^{2,D}(\zz^d)$ defined by
  $e_y(x)=\sqrt{1+|x|^{2D}}\uno{x=y}$ defines an orthonormal complete basis of
  $\ell^{2,D}(\zz^d)$ and, using the same fact as above,
  \[\sum_{y\in\zz^d}\|e_y\|^2_{2,2D}=\sum_{y\in\zz^d}\frac{1+|y|^{2D}}{1+|y|^{4D}}<\infty,\]
  so the embedding is Hilbert--Schmidt, and hence compact. \eqref{assum:rho:h}
  and \eqref{assum:rho:c} follow directly from the definition of the spaces in this case.

  \eqref{assum:2clt:1} and \eqref{assum:2clt:2} for $\rho^2_4$ are precisely what is
  assumed in Theorem \ref{thm:type:zzd}, and using this and Jensen's inequality we get the
  same estimates for $\rho^2_1$, $\rho^2_2$, and $\rho^2_3$. We are left checking
  \eqref{assum:2clt:3}. For simplicity we will assume here that $\Lambda\equiv0$. For
  \eqref{assum:2clt:3i}, the case $\ccz=\ell^\infty(\zz^d)$ is straightforward. Now if
  $\br{\mu_i}{1+|\cdot|^{8D}}<\infty$, $i=1,2$, and $\varphi\in\ell^{\infty,2D}(\zz^d)$,
  \begin{align}
    \left|\frac{J_{\mu_1,\mu_2}\varphi(z)}{1+|z|^{2D}}\right|
    &=\frac{1}{1+|z|^{2D}}\sum_{x\in\zz^d}\sum_{y\in\zz^d}\big(\varphi(y)-\varphi(x)\big)\Gamma(x,z,\{y\})\mu_1(\{x\})\\
    &\hspace{1in}+\frac{1}{1+|z|^{2D}}\sum_{x\in\zz^d}\sum_{y\in\zz^d}\big(\varphi(y)-\varphi(z)\big)\Gamma(z,x
    ,\{y\})\mu_2(\{x\})\\
    &\leq\frac{C\norm\varphi_{\infty,2D}}{1+|z|^{2D}}
    \Bigg[\sum_{x\in\zz^d}\sum_{y\in\zz^d}\left(1+|y|^{2D}\right)\,\Gamma(x,z,\{y\})\big(\mu_1(\{x\})+\mu_2(\{x\})\big)\\
    &\hspace{2.2in}+\sum_{x\in\zz^d}\left(1+|x|^{2D}\right)\mu_1(\{x\})+1+|z|^{2D}\Bigg]\\
    &\leq\frac{C\norm\varphi_{\infty,2d}}{1+|z|^{2D}}
    \Bigg[1+|z|^{2D}+\sum_{x\in\zz^d}\left(1+|x|^{2D}\right)\mu_2(\{x\})\Bigg]
    \leq C\norm\varphi_{\infty,2D}
  \end{align}
  uniformly in $z$, where we used \eqref{eq:mom:zzd:1} with a power of $2D$ instead of
  $8D$. We deduce that $\norm{J_{\mu_1,\mu_2}\varphi}_{\infty,2D}\leq
  C\norm\varphi_{\infty,2D}$ as required. The proof for $\ell^{\infty,3D}(\zz^d)$ is
  similar.  For \eqref{assum:2clt:3ii}, consider $\varphi\in\ell^\infty(\zz^d)$ and
  $\mu_1,\mu_2,\mu_3,\mu_4\in\cp$. Then
  \begin{align}
    \left|\left(J_{\mu_1,\mu_2}-J_{\mu_3,\mu_4}\right)\!\varphi(z)\right|
    &=\left|\int_W\!\Gamma\varphi(w;z)\big[\mu_1(dw)-\mu_3(dw)\big]
      +\int_W\!\Gamma\varphi(z;w)\big[\mu_2(dw)-\mu_4(dw)\big]\right|\\
    &\leq\norm{\Gamma\varphi(\cdot;z)}_\infty\norm{\mu_1-\mu_3}_{\ell^\infty(\zz^d)^\prime}
    +\norm{\Gamma\varphi(z;\cdot)}_\infty\norm{\mu_2-\mu_4}_{\ell^\infty(\zz^d)^\prime}.
  \end{align}
  Now
  $\norm{\Gamma\varphi(\cdot;z)}_\infty$ and $\norm{\Gamma\varphi(z;\cdot)}_\infty$ are
  both bounded by $4\olambda\norm\varphi_\infty$, so we get
  \begin{align}
    \norm{\left(J_{\mu_1,\mu_2}-J_{\mu_3,\mu_4}\right)\!\varphi}_\infty
    &\leq4\olambda\norm\varphi_\infty
    \left[\norm{\mu_1-\mu_3}_{\ell^{\infty,2d}(\zz^d)^\prime}
      +\norm{\mu_2-\mu_4}_{\ell^{\infty,2d}(\zz^d)^\prime}\right]
  \end{align}
  as required.
\end{proof}

\paragraph{Acknowledgements.} I am grateful to the three members of my PhD thesis
committee for their help during the course of this research. In particular, I thank Philip
Protter for suggesting to me the problem which originated this paper, Laurent Saloff-Coste
for helpful discussions about the functional analytical setting used in Theorem
\ref{thm:clt}, and my advisor, Rick Durrett, for his invaluable help throughout my
research and for his suggestions and comments on several versions of the manuscript. I
also want to thank an anonymous referee for a very careful reading of the manuscript and
for helpful comments and suggestions.

\bibliographystyle{natbib} \bibliography{biblio}

\begin{thebibliography}{}

\bibitem[Adams(2003)Adams]{adams}
Adams, R.~A. (2003).
\newblock {\em Sobolev Spaces\/}.
\newblock Academic Press, second edition.

\bibitem[Bayraktar {\em et~al.}(2007)Bayraktar, Horst, and Sircar]{bayHoSir}
Bayraktar, E., Horst, U., and Sircar, S. (2007).
\newblock Queueing theoretic approaches to financial price fluctuations.
\newblock {\em arXiv:math/0703832\/}.

\bibitem[Billingsley(1999)Billingsley]{bill}
Billingsley, P. (1999).
\newblock {\em Convergence of probability measures\/}.
\newblock Wiley Series in Probability and Statistics: Probability and
  Statistics. John Wiley \& Sons Inc., New York, second edition.

\bibitem[Blackwell and Dubins(1983)Blackwell and Dubins]{blackDub}
Blackwell, D. and Dubins, L.~E. (1983).
\newblock An extension of {S}korohod's almost sure representation theorem.
\newblock {\em Proc. Amer. Math. Soc.}, {\bf 89}(4), 691--692.

\bibitem[Dai~Pra {\em et~al.}(2007)Dai~Pra, Runggaldier, Sartori, and
  Tolotti]{daiPra}
Dai~Pra, P., Runggaldier, W.~J., Sartori, E., and Tolotti, M. (2007).
\newblock Large portfolio losses; a dynamic contagion model.
\newblock {\em arXiv:0704.1348v1 [math.PR]\/}.

\bibitem[Davis and Esparragoza-Rodriguez(2007)Davis and
  Esparragoza-Rodriguez]{davisEsp}
Davis, M. H.~A. and Esparragoza-Rodriguez, J.~C. (2007).
\newblock Large portfolio credit risk modeling.
\newblock {\em Int. J. Theor. Appl. Finance\/}, {\bf 10}(4), 653--678.

\bibitem[Duffie and Manso(2007)Duffie and Manso]{dufMan}
Duffie, D. and Manso, G. (2007).
\newblock Information percolation in large markets.
\newblock {\em American Economic Review, Papers and Proceedings\/}, {\bf 97},
  203--209.

\bibitem[Duffie and Sun(2007)Duffie and Sun]{dufSun}
Duffie, D. and Sun, Y. (2007).
\newblock Existence of independent random matching.
\newblock {\em Ann. Appl. Probab.}, {\bf 17}(1), 386--419.

\bibitem[Duffie {\em et~al.}(2005)Duffie, G{\^a}rleanu, and
  Pedersen]{dufGarPed}
Duffie, D., G{\^a}rleanu, N., and Pedersen, L.~H. (2005).
\newblock Over-the-counter markets.
\newblock {\em Econometrica\/}, {\bf 73}(6), 1815--1847.

\bibitem[Durrett(1996)Durrett]{durrett}
Durrett, R. (1996).
\newblock {\em Probability: theory and examples\/}.
\newblock Duxbury Press, Belmont, CA, second edition.

\bibitem[Ethier and Kurtz(1986)Ethier and Kurtz]{ethKur}
Ethier, S.~N. and Kurtz, T.~G. (1986).
\newblock {\em Markov processes: characterization and convergence\/}.
\newblock Wiley Series in Probability and Mathematical Statistics. John Wiley
  \& Sons Inc., New York.

\bibitem[Ferrari and Mari{\'c}(2007)Ferrari and Mari{\'c}]{ferMar}
Ferrari, P.~A. and Mari{\'c}, N. (2007).
\newblock Quasi stationary distributions and {F}leming-{V}iot processes in
  countable spaces.
\newblock {\em Electron. J. Probab.}, {\bf 12}, no. 24, 684--702 (electronic).

\bibitem[Fleming and Viot(1979)Fleming and Viot]{flViot}
Fleming, W.~H. and Viot, M. (1979).
\newblock Some measure-valued {M}arkov processes in population genetics theory.
\newblock {\em Indiana Univ. Math. J.}, {\bf 28}(5), 817--843.

\bibitem[F{\"o}llmer(1974)F{\"o}llmer]{foll}
F{\"o}llmer, H. (1974).
\newblock Random economies with many interacting agents.
\newblock {\em J. Math. Econom.}, {\bf 1}(1), 51--62.

\bibitem[Fournier and M{\'e}l{\'e}ard(2004)Fournier and
  M{\'e}l{\'e}ard]{fourMel}
Fournier, N. and M{\'e}l{\'e}ard, S. (2004).
\newblock A microscopic probabilistic description of a locally regulated
  population and macroscopic approximations.
\newblock {\em Ann. Appl. Probab.}, {\bf 14}(4), 1880--1919.

\bibitem[Giesecke and Weber(2004)Giesecke and Weber]{gieWeb}
Giesecke, K. and Weber, S. (2004).
\newblock Cyclical correlations, credit contagion, and portfolio losses.
\newblock {\em Journal of Banking and Finance\/}, {\bf 28}(12), 3009--3036.

\bibitem[Huck and Kosfeld(2007)Huck and Kosfeld]{huckKos}
Huck, S. and Kosfeld, M. (2007).
\newblock The dynamics of neighbourhood watch and norm enforcement.
\newblock {\em The Economic Journal\/}, {\bf 117}(516), 270--286.

\bibitem[Jacod and Shiryaev(1987)Jacod and Shiryaev]{jacodShir}
Jacod, J. and Shiryaev, A.~N. (1987).
\newblock {\em Limit theorems for stochastic processes\/}, volume 288 of {\em
  Grundlehren der Mathematischen Wissenschaften [Fundamental Principles of
  Mathematical Sciences]\/}.
\newblock Springer-Verlag, Berlin.

\bibitem[Joffe and M{\'e}tivier(1986)Joffe and M{\'e}tivier]{joffeMetiv}
Joffe, A. and M{\'e}tivier, M. (1986).
\newblock Weak convergence of sequences of semimartingales with applications to
  multitype branching processes.
\newblock {\em Adv. in Appl. Probab.}, {\bf 18}(1), 20--65.

\bibitem[Kufner(1980)Kufner]{kufner}
Kufner, A. (1980).
\newblock {\em Weighted Sobolev Spaces\/}.
\newblock Leipzig.

\bibitem[M\'el\'eard(1998)M\'el\'eard]{meleard}
M\'el\'eard, S. (1998).
\newblock Convergence of the fluctuations for interacting diffusions with jumps
  associated with {B}oltzmann equations.
\newblock {\em Stochastics Stochastics Rep.}, {\bf 63}(3-4), 195--225.

\bibitem[M{\'e}tivier(1987)M{\'e}tivier]{metiv}
M{\'e}tivier, M. (1987).
\newblock Weak convergence of measure valued processes using
  {S}obolev-imbedding techniques.
\newblock In {\em Stochastic partial differential equations and applications
  (Trento, 1985)\/}, volume 1236 of {\em Lecture Notes in Math.}, pages
  172--183. Springer, Berlin.

\bibitem[Revuz and Yor(1999)Revuz and Yor]{revYor}
Revuz, D. and Yor, M. (1999).
\newblock {\em Continuous martingales and {B}rownian motion\/}, volume 293 of
  {\em Grundlehren der Mathematischen Wissenschaften [Fundamental Principles of
  Mathematical Sciences]\/}.
\newblock Springer-Verlag, Berlin, third edition.

\bibitem[Yosida(1995)Yosida]{yosida}
Yosida, K. (1995).
\newblock {\em Functional analysis\/}.
\newblock Classics in Mathematics. Springer-Verlag, Berlin.
\newblock Reprint of the sixth (1980) edition.

\end{thebibliography}

\end{document}